\documentclass[final,onefignum,onetabnum]{siamart250211}    

\usepackage[framemethod=tikz]{mdframed}
\usepackage{amsmath}
\usepackage{amsfonts}
\usepackage{amssymb}
\usepackage{amsopn}
\usepackage{graphicx}
\usepackage[shortlabels]{enumitem}
\usepackage{lscape}
\usepackage{longtable}
\usepackage{rotating}
\usepackage{color}
\usepackage{url}
\usepackage{subfigure}
\usepackage{geometry}
\allowdisplaybreaks

\newtheorem{remark}{Remark}[section]

\newcommand{\R}{{\mathbb R}}
\newcommand{\N}{{\mathbb N}}
\newcommand{\sL}{{\mathbb L}}
\DeclareMathOperator{\argmin}{argmin}

\newcommand{\cE}{{\mathcal E}}
\newcommand{\cH}{{\mathcal H}}

\def\<{\langle}
\def\>{\rangle}

\setlist[enumerate]{label=$($\roman*$)$,leftmargin=3pt,align=left}

\begin{document}

\title{Recovering Nesterov accelerated dynamics from Heavy Ball dynamics via time rescaling\thanks{Submitted to the editors DATE.
\funding{RIB was partially supported by the Austrian Science Fund (FWF), project project P 34922-N. DAH was supported by the Doctoral Programme \emph{Vienna Graduate School on Computational Optimization (VGSCO)}, funded by the Austrian Science Fund (FWF), project W1260-N35. DKN’s research was funded by the Postdoctoral Scholarship Programme of Vingroup Innovation Foundation (VINIF), code VINIF.2024.STS.37.}}}

\headers{Recovering accelerated dynamics from Heavy Ball dynamics}{H. Attouch, R.I. Bo\c{t}, D.A. Hulett, and D.-K. Nguyen}

\author{Hedy Attouch\thanks{IMAG, Univ. Montpellier, CNRS, Montpellier, France, 
\email{hedy.attouch@umontpellier.fr}.} 
\and 
Radu Ioan Bo\c{t}\thanks{Faculty of Mathematics, University of Vienna, A-1090 Vienna, Austria, \email{radu.bot@univie.ac.at}.} 
\and David Alexander Hulett\thanks{Faculty of Mathematics, University of Vienna, Oskar-Morgenstern-Platz 1, 1090 Vienna, Austria, \email{david.alexander.hulett@univie.ac.at}.} 
\and 
Dang-Khoa Nguyen\thanks{Faculty of Mathematics and Computer Science, University of Science, Ho Chi Minh City, Vietnam, Vietnam National University, Ho Chi Minh City, Vietnam, \email{ndkhoa@hcmus.edu.vn}.}}

\maketitle

\begin{abstract}
In a real Hilbert space, we consider two classical problems: the global minimization of a smooth and convex function $f$ (i.e., a convex optimization problem) and finding the zeros of a monotone and continuous operator $V$ (i.e., a monotone equation). Attached to the optimization problem, first we study the asymptotic properties of the trajectories generated by a second-order dynamical system which features a constant viscous friction coefficient and a positive, monotonically increasing function $b(\cdot)$ multiplying $\nabla f$. When $b(\cdot)$ is identically $1$, we recover the Heavy Ball with friction dynamics introduced by Polyak in 1964. For a generated solution trajectory $y(t)$, we show small $o$ convergence rates dependent on $b(t)$ for $f(y(t)) - \min f$, and the weak convergence of $y(t)$ towards a global minimizer of $f$. In 2015, Su, Boyd and Candés introduced a second-order system which could be seen as the continuous-time counterpart of Nesterov's accelerated gradient. As the first key point of this paper, we show that for a special choice for $b(t)$, these two seemingly unrelated dynamical systems are connected: namely, they are time reparametrizations of each other. Every statement regarding the continuous-time accelerated gradient system may be recovered from its Heavy Ball counterpart. 

As the second key point of this paper, we observe that this connection extends beyond the optimization setting. Attached to the monotone equation involving the operator $V$, we again consider a Heavy Ball-like system which features an additional correction term which is the time derivative of the operator along the trajectory. We derive small $o$ rates for $\| V(y(t))\|$, which depend on two rescaling functions $\mu(\cdot)$ and $\gamma(\cdot)$, and show the weak convergence of $y(t)$ to a zero of $V$. For special choices for $\mu(\cdot)$ and $\gamma(\cdot)$, we establish a time reparametrization equivalence with the Fast OGDA dynamics introduced by Bo\c t, Csetnek and Nguyen in 2022, which can be seen as an analog of the continuous accelerated gradient dynamics, but for monotone operators. Again, every statement regarding the Fast OGDA system may be recovered from a Heavy Ball-like system. 
    \end{abstract}



\section{Introduction}\label{sec: section 1}
\subsection{Convex optimization}\label{subsec: subsec 11}
    
In a real Hilbert space $\mathcal{H}$, for a convex and continuously differentiable function $f : \mathcal{H} \to \R$, consider the minimization problem 
    \begin{equation}\label{eq: minimization problem}
        \min_{x\in \mathcal{H}} f(x),
    \end{equation}
which we assume to have an optimal solution. Su, Boyd and Candès \cite{SBC} observed that the second-order dynamical system
    \begin{equation}\label{eq: AVD}
        \ddot{x}(s) + \frac{\alpha}{s} \dot{x}(s) + \nabla f(x(s)) = 0 \quad \mbox{for} \ s \geq s_0 >0,
    \end{equation}
can be viewed as a continuous-time counterpart to Nesterov’s accelerated gradient method \cite{Nest1, Nest2}, designed to efficiently solve this problem. Attouch, Chbani, Peypouquet and Redont \cite{ACPR} and May \cite{May} showed that when $\alpha > 3$ any trajectory $x(s)$ generated by this dynamical system satisfies $f(x(s)) - \inf_{\mathcal{H}} f = o \left( \frac{1}{s^{2}}\right)$ and converges weakly to a global minimizer of $f$ as $s\to +\infty$. 

Also connected to the problem \eqref{eq: minimization problem}, we have the Heavy Ball dynamics first introduced by Polyak \cite{Polyak1,Polyak2}
    \begin{equation}
        \ddot{y}(t) + \lambda \dot{y}(t) + \nabla f(y(t)) = 0, 
    \end{equation}
where the name comes from the mechanical interpretation of this system: $y(t)$ describes the horizontal position of an object which moves alongside the graph of the function $f$ in a medium with viscous friction coefficient $\lambda$. For more details about this derivation, we refer the reader to \cite{AGR}. \'Alvarez \cite{Alv} first showed that any trajectory $y(t)$ generated by the Heavy Ball dynamics satisfies $f(y(t)) \to \inf_{\mathcal{H}} f$ and converges weakly to a global minimizer of $f$ as $t\to +\infty$. Later, it was proved that the functional values actually exhibit a rate $f(y(t)) - \inf_{\mathcal{H}} f = \mathcal{O}\left( \frac{1}{t}\right)$ as $t\to +\infty$. Observe how the rates for the functional values improve when instead of a constant friction coefficient $\lambda$, we have a so-called asymptotically vanishing damping $\frac{\alpha}{t}$ accompanying the velocity. 

In \cite{ACR-rescale}, Attouch, Chbani and Riahi analyzed a more general version of \eqref{eq: AVD} which reads 
    \begin{equation}\label{eq: time rescaled AVD, function, introduction}
        \ddot{x}(s) + \frac{\alpha}{s} \dot{x}(s) + b(s) \nabla f(x(s)) = 0,
    \end{equation}
    where $b : [s_{0}, +\infty[ \to \R_{++}$ is a differentiable and nondecreasing function. They explain how the presence of $b(\cdot)$ can be interpreted as the result of a time reparametrization of \eqref{eq: AVD}, which can yield faster convergence rates for the functional values provided $b(\cdot)$ is correctly chosen. Following the same idea, in Section \ref{sec: section 2} we will study the asymptotic properties of 
    \begin{equation}\label{eq: time rescaled Heavy Ball, function, introduction}
        \ddot{y}(t) + \lambda \dot{y}(t) + b(t) \nabla f(y(t)) = 0 \quad \mbox{for} \ t \geq t_0 \geq 0.
    \end{equation}
    We will show that if $b(\cdot)$ satisfies the growth condition $\sup_{t\geq t_{0}} \frac{\dot{b}(t)}{b(t)} < \lambda$, then any trajectory $y(t)$ generated by \eqref{eq: time rescaled Heavy Ball, function, introduction} satisfies $f(y(t)) - \inf_{\mathcal{H}} f = o\left( \frac{1}{\int_{\frac{t_{0} + t}{2}}^{t} b(r) dr}\right)$, and it converges weakly to a global minimizer of $f$ as $t\to +\infty$. Furthermore, and as one of the main points of this paper, in Section \ref{sec: connection between heavy ball and AVD, function case} we will show that for an appropriate choice for $b(\cdot)$, \eqref{eq: time rescaled Heavy Ball, function, introduction} and \eqref{eq: AVD} are time reparametrizations of each other. Every statement regarding \eqref{eq: AVD} may be recovered from properties of \eqref{eq: time rescaled Heavy Ball, function, introduction}. \vspace{1ex}
    
    \begin{mdframed}
            \begin{quote}
             	\centering
                Nesterov accelerated gradient dynamics can be recovered from the Heavy Ball dynamics through a time rescaling process.  
            \end{quote}
    \end{mdframed}
    
    \subsection{Monotone equations}\label{subsec: subsec 12}
    It turns out that this story goes beyond the optimization setting and has an analog for monotone equations. For a continuous and monotone operator $V : \mathcal{H} \to \mathcal{H}$, i.e. $\langle V(y) - V(x), y-x \rangle \geq 0$ for every $x,y \in \mathcal{H}$, consider the problem of finding the zeros of $V$. It is known that the gradient flow dynamics applied to $V$,
\begin{equation}\label{eq:monotoneflow}
        \dot{x}(s) + V(x(s)) = 0,
\end{equation}
fail in general to produce a trajectory which converges weakly to a zero of $V$, unless $V$ is cocoercive, i.e., for some $\beta >0$, it holds $\langle V(y) - V(x), y-x \rangle \geq \beta \|V(y) - V(x)\|^2$ for every $x,y \in \mathcal{H}$. Attouch, Bo\c t and Nguyen \cite{ABN} proved that in the latter case any trajectory $x(s)$ generated by \eqref{eq:monotoneflow} satisfies $\|V(x(s))\| = o\left(\frac{1}{\sqrt{s}}\right)$ and converges weakly to a zero of $V$ as $s \to +\infty$.
    
However, as shown by Attouch and Svaiter in \cite{AttouchSvaiter}, by adding a correction term which is the time derivative of the operator along the trajectory, i.e.,
\begin{equation}\label{eq: gradient descent + correction, operator}
        \lambda(s)\dot{x}(s) + \frac{d}{ds} V(x(s)) + V(x(s)) = 0,
    \end{equation}
where $\lambda(\cdot)$ is a positive function, then even in case of a monotone operator $V$ any trajectory $x(s)$ generated by this system satisfies $\| V(x(s))\| \to 0$ and weakly converges to a zero of $V$ as $s\to +\infty$. For $\lambda$ constant, the previous system can be rewritten equivalently as being of the form \eqref{eq:monotoneflow}, but applied to the Moreau-Yosida  envelope of $V$, which is always cocoercive. Combining second-order in time dynamics with an asymptotic vanishing damping of the form $\frac{\alpha}{s}$ has already been proven to produce fast convergence properties in the optimization setting; it is perhaps no too surprising that this effect can be transposed to the monotone inclusion setting. Indeed, for a general, possibly set-valued maximally monotone operator $V : \mathcal{H} \to 2^{\mathcal{H}}$, Attouch and Peypouquet studied in \cite{AP-max} the system
    \begin{equation}\label{eq: AVD for Moreau envelope, operator}
        \ddot{x}(s) + \frac{\alpha}{s} \dot{x}(s) + V_{\lambda(s)} (x(s)) = 0, \quad \text{where} \quad V_{\lambda} := \frac{1}{\lambda} \Bigl( \operatorname{Id} - (\operatorname{Id} + \lambda V)^{-1}\Bigr). 
    \end{equation}   
$V_{\lambda}$ is the aforementioned Moreau-Yosida envelope of $V$ of index $\lambda > 0$ of $V$. It is a known fact that $V$ and $V_{\lambda}$ share the same set of zeros.  For the above dynamics, if $\lambda(s)$ grows as $s^{2}$, the authors show a rate of convergence of $o\left( \frac{1}{s^{2}}\right)$ for $\| V_{\lambda(s)} (x(s))\|$, as well as the weak convergence of $x(s)$ to a zero of $V$ as $s\to +\infty$.

If $V$ is single-valued and continuous, we would prefer a scheme that evaluates $V(x(s))$ directly, rather than through its Moreau-Yosida envelope, also to have a setting that allows explicit discretizations and therefore forward evaluations of $V$, not as for \eqref{eq: AVD for Moreau envelope, operator}. Combining the ideas of having a correction term like in \eqref{eq: gradient descent + correction, operator} together with second-order and asymptotic vanishing terms like in \eqref{eq: AVD for Moreau envelope, operator} and a time rescaling coefficient similar to that of \eqref{eq: time rescaled AVD, function, introduction} gives rise to the Fast OGDA dynamics
    \begin{equation}\label{eq: fast OGDA, introduction}
        \ddot{x}(s) + \frac{\alpha}{s} \dot{x}(s) + \beta(s) \frac{d}{ds} V(x(s)) + \frac{1}{2} \left( \dot{\beta}(s) + \frac{\alpha}{s} \beta(s)\right) V(x(s)) = 0 \quad \mbox{for} \ s \geq s_0 >0,
    \end{equation}
introduced and studied by Bo\c t, Csetnek and Nguyen in \cite{fOGDA}. Provided that $\beta(\cdot)$ fulfills a growth condition, the authors show a rate of $o\left( \frac{1}{s \beta(s)}\right)$ for $\| V(x(s))\|$ and the weak convergence of $x(s)$ towards a zero of $V$ as $s\to +\infty$. For a more general system where a damping of the form $\frac{\alpha}{s^{r}}$ is considered, we refer the reader to \cite{fOGDA with r}. When $\beta(\cdot) \equiv 1$, the dynamics \eqref{eq: fast OGDA, introduction} reads 
    \begin{equation}\label{eq: fast OGDA, case beta(s)=1, introduction}
        \ddot{x}(s) + \frac{\alpha}{s} \dot{x}(s) + \frac{d}{ds} V(x(s)) + \frac{\alpha}{2s} V(x(s)) = 0. 
    \end{equation}
This is perhaps the most interesting case, since this system admits an explicit discretization which has identical convergence properties to its continuous-time counterpart, i.e., an algorithm which generates a sequence $(x_{k})_{k\in\N}$, combining Nesterov momentum with operator correction terms and using only forward evaluations of $V$,  and which fulfills $\| V(x_{k})\| = o\left( \frac{1}{k}\right)$ and converges weakly towards a zero of $V$ as $k\to +\infty$. 

In Section \ref{sec: Heavy ball system, operator}, we add an inertial term $\ddot{y}(t)$ to \eqref{eq: gradient descent + correction, operator} and we scale the terms $\frac{d}{dt} V(y(t))$ and $V(y(t))$ through positive functions $\mu(t)$ and $\gamma(t)$, which gives rise to the system
    \begin{equation}\label{eq: heavy ball system, operator, introduction}
        \ddot{y}(t) + \lambda \dot{y}(t) + \mu(t) \frac{d}{dt} V(y(t)) + \gamma(t) V(y(t)) = 0 \quad \mbox{for} \ t \geq t_0 \geq 0.
    \end{equation}
    Since we have a constant viscous friction coefficient $\lambda$ attached to $\dot{y}(t)$, this system can be seen as the Heavy Ball dynamics governed by a monotone and continuous operator $V$. Under a growth condition involving $\lambda$, $\mu{(\cdot)}$ and $\gamma{(\cdot)}$, we show that any trajectory $y(t)$ generated by this system fulfills $\| V(y(t))\| = o\left( \frac{1}{\mu(t)}\right)$ and converges weakly to a zero of $V$ as $t\to +\infty$. Additionally, and as the second main point of this paper, in Section \ref{sec: connection between heavy ball and AVD, operator case} we will show that for an appropriate choice of $\mu{(\cdot)}$ and $\gamma{(\cdot)}$, \eqref{eq: heavy ball system, operator, introduction} and \eqref{eq: fast OGDA, case beta(s)=1, introduction} are time reparametrizations of each other. Every statement regarding \eqref{eq: fast OGDA, case beta(s)=1, introduction} may be recovered from properties of \eqref{eq: heavy ball system, operator, introduction}.\vspace{1ex}
    
    \begin{mdframed}
            \begin{quote}
             	\centering
                The Fast OGDA dynamics can be recovered from the Heavy Ball dynamics for monotone equations through a time rescaling process.  
            \end{quote}
    \end{mdframed}

\section{Heavy Ball with friction dynamic from the time scaling perspective}\label{sec: section 2}

As mentioned in the introduction, for $\lambda > 0$ and $b \colon \left[ t_{0} , + \infty \right[ \to \R_{++}$ a differentiable and nondecreasing function we consider the following dynamical system
\begin{equation}
\label{HBF:eq}
\tag{\ensuremath{\mathrm{HBF}_{\lambda}}}
\ddot{y} \left( t \right) + \lambda \dot{y} \left( t \right) + b \left( t \right) \nabla f \left( y \left( t \right) \right) = 0 \quad \mbox{for} \ t \geq t_0 \geq 0,
\end{equation}
with Cauchy data $y \left( t_{0} \right) := y_{0}$ and $\dot{y} \left( t_{0} \right) := y_{1}$, where $y_{0}, y_{1} \in \cH$. 
It will turn out that the friction parameter $\lambda$ will play an important role when connecting \eqref{HBF:eq} to other dynamics.

We lay out this section as follows: in Subsection \ref{subsec: existence and uniqueness of solutions function case}, we show the existence and uniqueness of solutions to \eqref{HBF:eq}, and in Subsection \ref{subsec: convergence statements function case} we discuss their asymptotic properties. We now introduce the energy function that will help in our analysis, and show an initial bound that will be needed later. Assume that $x_{*} \in \argmin f$, the set of minimizers of $f$, and that $t\mapsto y(t)$ solves \eqref{HBF:eq} for $t\geq t_{0}$. For $0 \leq \eta \leq \lambda$ we define on $[t_0, +\infty)$
\begin{equation*}
\cE_{\eta} \left( t \right) := b \left( t \right) \left( f \left( y \left( t \right) \right) - \inf\nolimits_{\cH} f \right) + \dfrac{1}{2} \left\lVert \eta \left( y \left( t \right) - x_{*} \right) + \dot{y} \left( t \right) \right\rVert ^{2} + \dfrac{1}{2} \eta \left( \lambda - \eta \right) \left\lVert y \left( t \right) - x_{*} \right\rVert ^{2} .
\end{equation*}
The time derivative of $\cE_{\eta}(\cdot)$ at $t\geq t_{0}$ gives
\begin{align}
\dfrac{d}{dt} \cE_{\eta} \left( t \right) = \ & \dot{b} \left( t \right) \left( f \left( y \left( t \right) \right) - \inf\nolimits_{\cH} f \right) + b \left( t \right) \left\langle \nabla f \left( y \left( t \right) \right) , \dot{y} \left( t \right) \right\rangle \nonumber \\
& + \left\langle \eta \left( y \left( t \right) - x_{*} \right) + \dot{y} \left( t \right) , \eta \dot{y} \left( t \right) + \ddot{y} \left( t \right) \right\rangle + \eta \left( \lambda - \eta \right) \left\langle y \left( t \right) - x_{*} , \dot{y} \left( t \right) \right\rangle . \label{HBF:dE:pre}
\end{align}
According to the distributive property of the inner product and the equation \eqref{HBF:eq} we have 
\begin{align*}
& \left\langle \eta \left( y \left( t \right) - x_{*} \right) + \dot{y} \left( t \right) , \eta \dot{y} \left( t \right) + \ddot{y} \left( t \right) \right\rangle \nonumber \\
= \ 	& \left\langle \eta \left( y \left( t \right) - x_{*} \right) + \dot{y} \left( t \right) , \lambda \dot{y} \left( t \right) + \ddot{y} \left( t \right) \right\rangle + \left( \eta - \lambda \right) \left\langle \eta \left( y \left( t \right) - x_{*} \right) + \dot{y} \left( t \right) , \dot{y} \left( t \right) \right\rangle \nonumber \\
= \ 	& - b \left( t \right) \left\langle \eta \left( y \left( t \right) - x_{*} \right) + \dot{y} \left( t \right) , \nabla f \left( y \left( t \right) \right) \right\rangle + \eta \left( \eta - \lambda \right)  \left\langle y \left( t \right) - x_{*} , \dot{y} \left( t \right) \right\rangle + \left( \eta - \lambda \right) \left\lVert \dot{y} \left( t \right) \right\rVert ^{2} .
\end{align*}
By plugging this expression into \eqref{HBF:dE:pre}, we deduce that for every $t \geq t_0$
\begin{align}
\dfrac{d}{dt} \cE_{\eta} \left( t \right) & = \dot{b} \left( t \right) \left( f \left( y \left( t \right) \right) - \inf\nolimits_{\cH} f \right) - \eta b \left( t \right) \left\langle y \left( t \right) - x_{*} , \nabla f \left( y \left( t \right) \right) \right\rangle + \left( \eta - \lambda \right) \left\lVert \dot{y} \left( t \right) \right\rVert ^{2} \nonumber \\
& \leq \left( \dot{b} \left( t \right) - \eta b \left( t \right) \right) \left( f \left( y \left( t \right) \right) - \inf\nolimits_{\cH} f \right) + \left( \eta - \lambda \right) \left\lVert \dot{y} \left( t \right) \right\rVert ^{2} , \label{HBF:dE:inq}
\end{align}
where the last inequality comes from the convexity of $f$.

\subsection{Existence and uniqueness of global solutions}\label{subsec: existence and uniqueness of solutions function case}

We may rewrite \eqref{HBF:eq} as a first order system in two variables. It is simple to check that 
\begin{equation}\label{HBF:eq with initial conditions}
   \tag{\ensuremath{\mathrm{HBF}_{\lambda}}}
        \begin{cases}
            \ddot{y}(t) + \lambda \dot{y}(t) + b(t) \nabla f(y(t)) = 0, \\
            y(t_{0}) = y_{0}, \quad \dot{y}(t_{0}) = y_{1}
        \end{cases}
    \end{equation}
is equivalent to 
\[
    \begin{cases}
        \dot{y}(t) &= {u}(t) - \lambda y(t), \\
        \dot{{u}}(t) &= -b(t) \nabla f(y(t)), \\
        y(t_{0}) &= y_{0}, \\
        {u}(t_{0}) &= \lambda y_{0} + y_{1}, 
    \end{cases}
\]
where 
\[
{u}(t) := \lambda y(t) + \dot{y}(t). 
\]
We may write the system above in a more compact way, namely,
\begin{equation}\label{eq: HBF function reformulated as first order}
    \begin{cases}
        \bigl( \dot{y}(t), \dot{{u}}(t)\bigr) &= G(t, y(t), {u}(t)) \\
        (y(t_{0}), {u}(t_{0})) &= \bigl(y_{0}, \: \lambda y_{0} + y_{1}\bigr), 
    \end{cases}
\end{equation}
where $G :[t_{0}, +\infty[ \times \mathcal{H} \times \mathcal{H} \to \mathcal{H} \times \mathcal{H}$ is given by 
\[
    G(t, y, {u}) := \Bigl({u} - \lambda y, \: -b(t) \nabla f(y)\Bigr). 
\]
In order to establish uniqueness and global existence of the solution trajectory, we require the following assumption.\vspace{1ex}
\begin{mdframed}
Suppose that $\lambda > 0$ and $b \colon \left[ t_{0} , + \infty \right[ \to \R_{++}$ additionally satisfy
\begin{equation}
\label{HBF:b-lambda}
\sup\limits_{t \geq t_{0}} \dfrac{\dot{b} \left( t \right)}{b \left( t \right)} < \lambda .
\end{equation}
\end{mdframed}
\begin{theorem}
Suppose further that $\nabla f$ is Lipschitz continuous on bounded sets, $b \colon \left[ t_{0} , + \infty \right[ \to \R_{++}$ is continuously differentiable, and that $\lambda$ and $b(\cdot)$ satisfy assumption \eqref{HBF:b-lambda}. Then, the dynamical system \eqref{HBF:eq with initial conditions} admits a unique global solution $y : [t_{0}, +\infty[ \to \mathcal{H}$. 
\end{theorem}
\begin{proof}
Since $b(\cdot)$ is continuously differentiable, it is Lipschitz continuous on the bounded subsets of $[t_{0}, +\infty[$. This, together with our other assumptions ensures that $G$ is Lipschitz continuous on the bounded subsets of $[t_{0}, +\infty[ \times\mathcal{H} \times \mathcal{H}$. According to \cite[Theorems 46.2 and 46.3]{SellYou}, the ordinary differential equation \eqref{eq: HBF function reformulated as first order} admits a unique continuously differentiable solution $t \mapsto (y(t), {u}(t))$ defined on an interval $[t_{0}, T_{\text{max}}[$, where 
\[
\text{either} \quad T_{\text{max}} = +\infty \quad \text{or} \quad \lim_{t\to T_{\text{max}}} \| (y(t), {u}(t))\| = +\infty. 
\]
According to \eqref{HBF:b-lambda}, there exists $0 < \eta_{0} < \lambda$ such that
\[
\frac{\dot{b}(t)}{b(t)} \leq \eta_{0} < \lambda \quad \forall t\geq t_{0}.
\]
We may now define the energy functional $\mathcal{E}_{\eta_{0}}(\cdot)$ restricted to the interval $[t_{0}, T_{\text{max}}[$. According to the previous inequality and \eqref{HBF:dE:inq}, we have for every $t \in [t_{0}, T_{\text{max}}[$
\[
\frac{d}{dt} \mathcal{E}_{\eta_{0}}(t) \leq \left( \sup_{t\geq t_{0}} \frac{\dot{b}(t)}{b(t)} - \eta_{0}\right) b(t) \Bigl( f(y(t)) - \inf\nolimits_{\mathcal{H}}f\Bigr) + (\eta_{0} - \lambda) \bigl\| \dot{y}(t)\bigr\|^{2} \leq 0 .
\]
This means that $\mathcal{E}_{\eta_{0}}(\cdot)$ is nonincreasing on $[t_{0}, T_{\text{max}}[$. In particular,  we obtain that for every $t\in [t_{0}, T_{\text{max}}[$ it holds 
\[
\frac{1}{2} \bigl\| \eta_{0}(y(t) - x_{*}) + \dot{y}(t)\bigr\|^{2} + \frac{1}{2} \eta_{0}(\lambda - \eta_{0}) \| y(t) - x_{*}\|^{2} \leq \mathcal{E}_{\eta_{0}}(t) \leq \mathcal{E}_{\eta_{0}}(t_{0}). 
\]
This immediately yields that $t \mapsto \|y(t) - x_{*} \|$ and thus $t \mapsto  \|y(t) \|$ are bounded on $[t_{0}, T_{\text{max}}[$. Since we also obtain that $t \mapsto \| \eta_{0} (y(t) - x_{*}) + \dot{y}(t)\|$ is bounded on $[t_{0}, T_{\text{max}}[$, using the triangle inequality, we obtain that $t \mapsto \bigl\| \dot{y}(t)\bigr\|$ and thus $t \mapsto \|{u}(t)\|$ are bounded on $[t_{0}, T_{\text{max}}[$. So it must be $T_{\text{max}} = +\infty$, which completes the proof of this theorem. 
\end{proof}

\subsection{Convergence rates for function values and weak convergence of trajectories}\label{subsec: convergence statements function case}

We recall that $f: \cH \rightarrow \R$ is assumed to be convex and continuously differentiable function, $\lambda >0$ and $b \colon [t_{0}, +\infty[ \to \R_{++}$ is a differentiable and nondecreasing function. We will now assume that we have a global solution $y \colon [t_{0}, +\infty[ \to \mathcal{H}$ to \eqref{HBF:eq}. 

\begin{proposition}
\label{prop:HBF}
Suppose further that $\lambda$ and $b(\cdot)$ satisfy assumption \eqref{HBF:b-lambda}.
Let $y \colon \left[ t_{0} , + \infty \right[ \to \cH$ be a solution trajectory of \eqref{HBF:eq}.
Then, the following statements are true: 
\begin{enumerate}
\item \label{prop:HBF:int} \emph{(integrability results)}
It holds
\begin{equation*}
\int_{t_{0}}^{+ \infty} b \left( t \right) \left( f \left( y \left( t \right) \right) - \inf\nolimits_{\cH} f \right) dt < + \infty \quad \mbox{and} \quad
\int_{t_{0}}^{+ \infty} \left\lVert \dot{y} \left( t \right) \right\rVert ^{2} dt < + \infty .
\end{equation*}

\item \label{prop:HBF:lim} \emph{(energy functions convergence)}
The limit $\lim_{t \to + \infty} \cE_{\eta} \left( t \right) \in \R$ exists for every $\eta$ satisfying $0 \leq \eta \leq \lambda$.
\end{enumerate}
\end{proposition}
\begin{proof}
\begin{enumerate}
\item 
According to \eqref{HBF:b-lambda}, there exists $\eta_{0} > 0$ such that
\begin{equation*}
\sup\limits_{t \geq t_{0}} \dfrac{\dot{b} \left( t \right)}{b \left( t \right)} < \eta_{0} < \lambda .
\end{equation*}

It follows from \eqref{HBF:dE:inq} that for every $t \geq t_{0}$
\begin{equation}\label{eq: time derivative of energy functional is nonpositive, function case}
\dfrac{d}{dt} \cE_{\eta_{0}} \left( t \right) \leq \left( \sup\limits_{t \geq t_{0}} \dfrac{\dot{b} \left( t \right)}{b \left( t \right)} - \eta_{0} \right) b \left( t \right) \left( f \left( y \left( t \right) \right) - \inf\nolimits_{\cH} f \right) + \left( \eta_{0} - \lambda \right) \left\lVert \dot{y} \left( t \right) \right\rVert ^{2} .
\end{equation}
The statement follows upon integration of this inequality.

\item 
Let $0 \leq \eta \leq \lambda$. From \eqref{HBF:dE:inq} we derive that for every $t \geq t_{0}$
\begin{align*}
\dfrac{d}{dt} \cE_{\eta} \left( t \right) & \leq (\lambda-\eta)b \left( t \right) \left( f \left( y \left( t \right) \right) - \inf\nolimits_{\cH} f \right) .
\end{align*}
The first statement in  \ref{prop:HBF:int} ensures that the right hand side of this estimate belongs to $\sL^{1} (\left[ t_{0} , + \infty \right[; \R)$.
Hence, the conclusion follows from Lemma \ref{lem:lim-R}.
\end{enumerate}
\end{proof}

For the purpose of studying the convergence and rate of convergence of  \eqref{HBF:eq}, we introduce the following function $W \colon \left[ t_{0} , + \infty \right[ \to \R_{+}$ defined for every $t \geq t_{0}$ as
\begin{equation*}
W \left( t \right) := \left( f \left( y \left( t \right) \right) - \inf\nolimits_{\cH} f \right) + \dfrac{1}{2b \left( t \right)} \left\lVert \dot{y} \left( t \right) \right\rVert ^{2} .
\end{equation*}
The function is nonincreasing, which plays a crucial role in establishing convergence rates  (see \cite{BC}). Indeed, for every $t \geq t_{0}$ we have
\begin{equation}\label{eq: decreasing property of W}
    \dot{W} \left( t \right) = \left\langle \nabla f \left( y \left( t \right) \right) , \dot{y} \left( t \right) \right\rangle - \dfrac{\dot{b} \left( t \right)}{2b^{2} \left( t \right)} \left\lVert \dot{y} \left( t \right) \right\rVert ^{2} + \dfrac{1}{b \left( t \right)} \left\langle \dot{y} \left( t \right) , \ddot{y} \left( t \right) \right\rangle = - \dfrac{1}{b \left( t \right)} \left( \dfrac{\dot{b} \left( t \right)}{2b \left( t \right)} + \lambda \right) \left\lVert \dot{y} \left( t \right) \right\rVert ^{2} \leq 0 ,
\end{equation}
which holds due to the fact that $\dot{b} \left( t \right) \geq 0$ for every $t \geq t_{0}$. We are now ready for the main convergence results for \eqref{HBF:eq}.
\begin{theorem}
\label{thm:HBF}
Suppose further that $\lambda$ and $b(\cdot)$ satisfy assumption \eqref{HBF:b-lambda}. Let $y \colon \left[ t_{0} , + \infty \right[ \to \cH$ be a solution trajectory of \eqref{HBF:eq}. Then, the following statements are true: 
\begin{enumerate}
\item \label{thm:HBF:rate} (convergence rates)
It holds
\begin{equation*}
W \left( t \right) = o \left( \dfrac{1}{\int_{\frac{t+t_{0}}{2}}^{t} b \left( r \right) dr} \right) \quad \textrm{ as } t \to + \infty .
\end{equation*}
In particular,
\begin{equation*}
f \left( y \left( t \right) \right) - \inf\nolimits_{\cH} f = o \left( \dfrac{1}{\int_{\frac{t+t_{0}}{2}}^{t} b \left( r \right) dr} \right)
\quad \textrm{ and } \quad
\left\lVert \dot{y} \left( t \right) \right\rVert = o \left( \sqrt{\dfrac{b \left( t \right)}{\int_{\frac{t+t_{0}}{2}}^{t} b \left( r \right) dr}} \right) \quad \textrm{ as } t \to + \infty .
\end{equation*}

\item \label{thm:HBF:tra} (trajectory convergence)
The solution trajectory $y(t)$ converges weakly to an element of $\argmin f$ as $t\to +\infty$.
\end{enumerate}
\end{theorem}
\begin{proof}
\begin{enumerate}
\item 
From Proposition \ref{prop:HBF}, we infer that
\begin{equation*}
\int_{t_{0}}^{+ \infty} b \left( t \right) W \left( t \right) dt < + \infty .
\end{equation*}
In particular, for $\psi \colon \left[ t_{0} , + \infty \right[ \to \R_{+}$ defined as
\begin{equation*}
\psi \left( t \right) := \int_{t_{0}}^{t} b \left( r \right) W \left( r \right) dr ,
\end{equation*}
the limit $\lim_{t \to + \infty} \psi \left( t \right) \in \R$ exists.	
Hence, for every $t \geq t_{0}$ and $r\in \left[ \frac{t_{0} + t}{2}, t\right]$ we have, according to the decreasing property of $W(\cdot)$, that $W(t) \leq W(r)$ and thus
\begin{align*}
0 \leq W \left( t \right) \int_{\frac{t+t_{0}}{2}}^{t} b \left( r \right) dr \leq \int_{\frac{t+t_{0}}{2}}^{t} b \left( r \right) W \left( r \right) dr = \psi \left( t \right) - \psi \left( \dfrac{t+t_{0}}{2} \right) \to 0 
\quad \textrm{ as } t \to + \infty,
\end{align*}
which {gives the desired small $o$ rate for $W(\cdot)$}.

\item 
Let $0 < \eta_{1} < \eta_{2} < \lambda$ and $x_* \in \argmin f$.
We have for every $t \geq t_{0}$
\begin{align*}
\cE_{\eta_{2}} \left( t \right) - \cE_{\eta_{1}} \left( t \right) & = \dfrac{1}{2} \left( \eta_{2} - \eta_{1} \right) \lambda \left\lVert y \left( t \right) - x_{*} \right\rVert ^{2} + \left( \eta_{2} - \eta_{1} \right) \left\langle y \left( t \right) - x_{*} , \dot{y} \left( t \right) \right\rangle \nonumber \\
& = \left( \eta_{2} - \eta_{1} \right) \left( \dfrac{1}{2} \lambda \left\lVert y \left( t \right) - x_{*} \right\rVert ^{2} + \dfrac{d}{dt} \left( \dfrac{1}{2} \left\lVert y \left( t \right) - x_{*} \right\rVert ^{2} \right) \right) .
\end{align*}
Proposition \ref{prop:HBF} \ref{prop:HBF:lim} guarantees that $\lim_{t \to + \infty} \left( \cE_{\eta_{2}} \left( t \right) - \cE_{\eta_{1}} \left( t \right) \right) \in \R$ exists. If we set $q(t) := \frac{1}{2} \| y(t) - x_{*}\|^{2}$, Lemma \ref{lem:q} ensures that $\lim_{t \to + \infty} \frac{1}{2} \left\lVert y \left( t \right) - x_{*} \right\rVert ^{2}$ exists and is a real number. This shows that the first condition of the Opial Lemma (see Lemma \ref{Opial}) is verified. Moreover, the second condition is verified due to the weak lower semicontinuity of $f$ and the convergence of $f(y(t))$ to $\inf_{\cH} f$ as $t \to +\infty$, as established in (i). 
\end{enumerate}	
\end{proof}

Let us illustrate the preceding results with two specific choices. For $\kappa >0$ and $\rho \geq 0$, consider 
$$b \left( t \right) := {\kappa} \exp \left({\rho} t \right) \quad \mbox{and} \quad b \left( t \right) := {\kappa} t^{\rho},$$
respectively. Note that in both cases the classical Heavy Ball Method with friction is recovered by setting $(\kappa, \rho) = (1, 0)$, yielding a convergence rate of $o(1/t)$ for the function value along the solution trajectory.  This improves the $\mathcal{O} \left( 1/t \right)$ convergence rate results derived in \cite{Sun}.  For these choices of $b(\cdot)$ we have for $t\geq t_{0}$ sufficiently large
\begin{align*}
    \int_{\frac{t + t_{0}}{2}}^{t} b(r) dr &= \frac{\kappa}{\rho} \left[\exp(\rho t) - \exp\left( \frac{\rho (t + t_{0})}{2}\right)\right] = \frac{\kappa}{\rho} \exp(\rho t) \left[ 1 - \exp\left( \frac{\rho (t_{0} - t)}{2}\right)\right] \geq \frac{\kappa}{2 \rho} \exp(\rho t) \\
\intertext{and}
    \int_{\frac{t + t_{0}}{2}}^{t} b(r) dt &= \frac{\kappa}{\rho + 1} \left[t^{\rho + 1} - \left( \frac{t + t_{0}}{2}\right)^{\rho + 1}\right] = \frac{\kappa}{\rho + 1} t^{\rho + 1} \left[ 1 - \left(\frac{t + t_{0}}{2t}\right)^{\rho + 1}\right] \geq \frac{\kappa}{4(\rho + 1)} t^{\rho + 1},
\end{align*}
respectively. These estimates lead to the following corollaries, respectively.

\begin{corollary}\label{coro:exp}
Let $y \colon \left[ t_{0} , + \infty \right[ \to \cH$ be a solution trajectory of
\begin{equation*}
\ddot{y} \left( t \right) + \lambda \dot{y} \left( t \right) + {\kappa} \exp \left({\rho} t \right) \nabla f \left( y \left( t \right) \right) = 0 \quad \mbox{for} \ t \geq t_0 \geq 0.
\end{equation*}
If $\lambda > \rho > 0$, then the following statements are true:
\begin{enumerate}[\rm (i)]
\item (convergence rates)
It holds
\begin{equation*}
    f \left( y \left( t \right) \right) - \inf\nolimits_{\cH} f = o \left( \frac{1}{\exp(\rho t)}\right)
    \quad \textrm{ and } \quad
    \left\lVert \dot{y} \left( t \right) \right\rVert \to 0
    \quad 
    \text{as $t\to +\infty$}.
\end{equation*}

\item (trajectory convergence)
The solution trajectory $y(t)$ converges weakly to an element of $\argmin f$ as $t\to +\infty$.
\end{enumerate}	
\end{corollary}	

\begin{corollary}
Let $y \colon \left[ t_{0} , + \infty \right[ \to \cH$ be a solution trajectory of
\begin{equation*}
\ddot{y} \left( t \right) + \lambda \dot{y} \left( t \right) + \kappa t^{\rho} \nabla f \left( y \left( t \right) \right) = 0 \quad \mbox{for} \ t \geq t_0 > 0.
\end{equation*}
If $\lambda > \dfrac{\rho}{t_0} \geq 0$, then the following statements are true:
\begin{enumerate}
\item (convergence rates)
It holds
\begin{equation*}
f \left( y \left( t \right) \right) - \inf\nolimits_{\cH} f = o \left( \dfrac{1}{t^{\rho+1}} \right)
\quad \textrm{ and } \quad
\left\lVert \dot{y} \left( t \right) \right\rVert = o \left( \dfrac{1}{\sqrt{t}} \right) \qquad \textrm{ as } t \to + \infty .
\end{equation*}

\item (trajectory convergence)
The solution trajectory $y(t)$ converges weakly to an element of $\argmin f$ as $t\to +\infty$.
\end{enumerate}	
\end{corollary}	

\begin{remark}
Previous statements show that one can get faster rates from the Heavy Ball Method (even linearly, as in Collorary \ref{coro:exp}). Nevertheless, we emphasize that these dynamics are not only interesting in terms of their acceleration phenomenon but also because they have a connection with other dynamics in the literature, for {a} nonclassical choice of $b(\cdot)$. This was one of the main points made in the introduction, which we elaborate upon in the next section.
\end{remark}

\section{Connection with the dynamical system with asymptotic vanishing damping}\label{sec: connection between heavy ball and AVD, function case}

For $\alpha > 3$, we study the convergence behavior of the second order dynamical system with asymptotic vanishing damping
\begin{equation}
\label{AVD:eq}
\tag{\ensuremath{\mathrm{AVD}_{\alpha}}}
\ddot{x} \left( s \right) + \dfrac{\alpha}{s} \dot{x} \left( s \right) + \nabla f \left( x \left( s \right) \right) = 0 \quad \mbox{for} \quad s \geq s_0 >0.
\end{equation}
In particular, we will show that \eqref{AVD:eq} can be derived from \eqref{HBF:eq} via a time rescaling argument. This connection allows us to transfer the convergence results established for \eqref{HBF:eq} to \eqref{AVD:eq}.

\subsection{Two equivalent dynamical systems through time rescaling}
 
We start with a solution trajectory $y :[t_{0}, +\infty[ \to \mathcal{H}$ of
    \begin{equation}\label{eq: heavy ball system function, for the time rescaling argument}
        \ddot{y}(t) + \lambda \dot{y}(t) + b(t) \nabla f(y(t)) = 0, 
    \end{equation}
and define $x(s) := y(\tau(s))$, where $\tau :[s_{0}, +\infty[ \to [t_{0}, +\infty[$ is a continuously differentiable function such that $\dot{\tau}(s) > 0$ for every $s\geq s_{0} > 0$ and $\lim_{s \to + \infty} \tau(s) = + \infty$. We have 
    \[
        \dot{x}(s) = \dot{\tau}(s) \dot{y}(\tau(s)) \quad \text{and} \quad \ddot{x}(s) = \ddot{\tau}(s) \dot{y}(\tau(s)) + \bigl( \dot{\tau}(s)\bigr)^{2} \ddot{y}(\tau(s)).
    \]
These expressions lead to 
    \[
        \dot{y}(\tau(s)) = \frac{1}{\dot{\tau}(s)} \dot{x}(s) \quad \text{and} \quad \ddot{y}(\tau(s)) = \frac{1}{\bigl( \dot{\tau}(s)\bigr)^{2}} \Bigl[ \ddot{x}(s) - \ddot{\tau}(s) \dot{y}(\tau(s))\Bigr] = \frac{1}{\bigl( \dot{\tau}(s)\bigr)^{2}} \Bigl[ \ddot{x}(s) - \frac{\ddot{\tau}(s)}{ \dot{\tau}(s)} \dot{x}(s)\Bigr]. 
    \]
Now, plugging $t = \tau(s)$ in \eqref{eq: heavy ball system function, for the time rescaling argument} for $s \geq s_0$ gives
    \[
        \frac{1}{\bigl( \dot{\tau}(s)\bigr)^{2}} \left[ \ddot{x}(s) - \frac{\ddot{\tau}(s)}{\dot{\tau}(s)} \dot{x}(s)\right] + \frac{\lambda}{\dot{\tau}(s)} \dot{x}(s) + b(\tau(s)) \nabla f(x(s)) = 0
    \]
or, equivalently,
    \begin{equation}\label{eq: time rescaled heavy ball, function}
        \ddot{x}(s) + \left[ \lambda \dot{\tau}(s) - \frac{\ddot{\tau}(s)}{\dot{\tau}(s)}\right] \dot{x}(s) + \bigl( \dot{\tau}(s)\bigr)^{2} b(\tau(s)) \nabla f(x(s)) = 0. 
    \end{equation}
Recall that the Su-Boyd-Candès dynamics \cite{SBC} are given by 
    \[
        \ddot{x}(s) + \frac{\alpha}{s} \dot{x}(s) + \nabla f(x(s)) = 0. 
    \]
Going back to \eqref{eq: time rescaled heavy ball, function}, to match the asymptotic vanishing viscosity accompanying the velocity we need the following to hold
    \[
        \begin{cases}
            \displaystyle\lambda \dot{\tau}(s) - \frac{\ddot{\tau}(s)}{ \dot{\tau}(s)} &= \displaystyle\frac{\alpha}{s}, \\
            \tau(s_{0}) &= t_{0}. 
        \end{cases}
    \]
It is straightforward to check that 
    \[
        \tau(s) := \frac{\alpha - 1}{\lambda} \ln\left( \frac{s}{s_{0}}\right) + t_{0}
    \]
satisfies the previous differential equation. We have 
    \[
        \dot{\tau}(s) = \frac{\alpha - 1}{\lambda s} \quad \mbox{and} \quad \bigl( \dot{\tau}(s)\bigr)^{2} = \frac{(\alpha - 1)^{2}}{\lambda^{2} s^{2}}. 
    \]
We need, of course, to assume that $\alpha > 1$. Furthermore, we wish the coefficient attached to $\nabla f(x(s))$ to be $1$, i.e., 
    \[
        \bigl( \dot{\tau}(s)\bigr)^{2} b(\tau(s)) = 1 \quad \Leftrightarrow \quad b\left( \frac{\alpha - 1}{\lambda} \ln\left( \frac{s}{s_{0}}\right) + t_{0}\right) = \frac{\lambda^{2} s^{2}}{(\alpha - 1)^{2}}, 
    \]
which is fulfilled if we choose
    \[
        b(t) = \left(\frac{\lambda s_{0}}{\alpha - 1}\right)^{2} \exp\left( \frac{2\lambda (t  - t_{0})}{\alpha - 1}\right).
    \]
Furthermore, we need $b(\cdot)$ to satisfy \eqref{HBF:b-lambda}. Indeed, we have 
    \[
        \frac{\dot{b}(t)}{b(t)} = \frac{\left( \frac{2\lambda}{\alpha - 1}\right) \exp\left( \frac{2\lambda (t - t_{0})}{\alpha - 1}\right)}{\exp \left( \frac{2\lambda(t - t_{0})}{\alpha - 1}\right)} = \frac{2\lambda}{\alpha - 1} < \lambda \quad \Leftrightarrow \quad \alpha > 3. 
    \]
All in all, with these choices, the trajectory $t \mapsto x(t)$ fulfills 
    \[
        \ddot{x}(s) + \frac{\alpha}{s} \dot{x}(s) + \nabla f(x(s)) = 0. 
    \]
Conversely, if for $\alpha > 3$, $x : [s_{0}, +\infty[ \to \mathcal{H}$ is a solution to the previous system and we define $y(t) = x(\sigma(t))$, where $\sigma : [t_{0}, +\infty[ \to [s_{0}, +\infty[$ is a continuously differentiable function such that $\dot{\sigma}(t) > 0$ for all $t\geq t_{0} > 0$ and $\lim_{t \to + \infty} \sigma(t) = + \infty$, arguing in a similar fashion as it was done previously we arrive at 
    \[
        \ddot{y}(t) + \left[ \alpha \frac{\dot{\sigma}}{\sigma(t)} - \frac{\ddot{\sigma}(t)}{\dot{\sigma}(t)}\right] \dot{y}(t) + \bigl( \dot{\sigma}(t)\bigr)^{2} \nabla f(y(t)) = 0. 
    \]
We want the coefficient attached to $\dot{y}(t)$ to be $\lambda$, i.e., we want $\sigma(\cdot)$ to satisfy the differential equation 
    \[
        \begin{cases}
            \displaystyle\alpha\frac{\dot{\sigma}(t)}{\sigma(t)} - \frac{\ddot{\sigma}(t)}{\dot{\sigma}(t)} &= \lambda, \\
            \sigma(t_{0}) &= s_{0},
        \end{cases}
    \]
which is fulfilled by 
    \[
        \sigma(t) := s_{0} \exp\left( \frac{\lambda(t - t_{0})}{\alpha - 1}\right).
    \]
With this choice for $\sigma(\cdot)$, the resulting system reads 
    \[
        \ddot{y}(t) + \lambda \dot{y}(t) + s_{0}^{2} \exp\left( \frac{2\lambda(t - t_{0})}{\alpha - 1}\right) \nabla f(y(t)) = 0. 
    \]
We have thus showed the following statement.
    
    \begin{proposition}\label{prop: Heavy Ball and AVD connected, function}
        Assume that $\alpha > 3$, $\lambda > 0$ and that $s_0 > 0, t_0 \geq 0$ are initial times. Consider the following second-order systems: 
        \begin{equation}\label{eq: heavy ball and avd connected, function case, heavy ball}
            \begin{cases}
                \ddot{y}(t) + \lambda \dot{y}(t) + s_{0}^{2} \exp\left( \frac{2\lambda (t - t_{0})}{\alpha - 1}\right) \nabla f(y(t)) = 0, \\
                y(t_{0}) = y_{0}, \quad \dot{y}(t_{0}) = y_{1}, 
            \end{cases}
        \end{equation}
        and
        \begin{equation}\label{eq: heavy ball and avd connected, function case, avd}
            \begin{cases}
                \ddot{x}(s) + \frac{\alpha}{s} \dot{x}(s) + \nabla f(x(s)) = 0, \\
                x(s_{0}) = x_{0}, \quad \dot{x}(s_{0}) = x_{1}.
            \end{cases}
        \end{equation}
Then, the following statements are true:  

    \begin{enumerate}[\rm (i)]
        \item If $y : [t_{0}, +\infty) \to \mathcal{H}$ is a solution trajectory of \eqref{eq: heavy ball and avd connected, function case, heavy ball} and the function $\tau : [s_{0}, +\infty) \to [t_{0}, +\infty)$ is given by 
        \[
            \tau(s) := \frac{\alpha - 1}{\lambda} \ln\left( \frac{s}{s_{0}}\right) + t_{0},
        \]  
        then the reparametrized trajectory $x : [s_{0}, +\infty) \to \mathcal{H}$ given by $x(s) := y(\tau(s))$ is a solution of \eqref{eq: heavy ball and avd connected, function case, avd} for initial conditions 
        \[
            x(s_{0}) = y_{0} \quad \text{and} \quad \dot{x}(s_{0}) = \frac{\alpha - 1}{\lambda s_{0}} y_{1}.
        \]
        \item If $x : [s_{0}, +\infty) \to \mathcal{H}$ is a solution trajectory to \eqref{eq: heavy ball and avd connected, function case, avd} and the function $\sigma : [t_{0}, +\infty) \to [s_{0}, +\infty)$ is given by 
        \[
            \sigma(t) := s_{0} \exp\left( \frac{\lambda (t - t_{0})}{\alpha - 1}\right),
        \]
        then the reparametrized trajectory $y: [t_{0}, +\infty) \to \mathcal{H}$ given by $y(t) := x(\sigma(t))$ is a solution of \eqref{eq: heavy ball and avd connected, function case, heavy ball} for initial conditions 
        \[
            y(t_{0}) = x_{0} \quad \text{and} \quad \dot{y}(t_{0}) = \frac{\lambda s_{0}}{\alpha - 1} x_{1}.
        \]
    \end{enumerate}
\end{proposition}
    
\subsection{Transferring the convergence results to the  \texorpdfstring{\eqref{AVD:eq}}{} framework}
    As a direct corollary of Theorem \ref{thm:HBF} and Proposition \ref{prop: Heavy Ball and AVD connected, function}, we obtain the following theorem.
    \begin{theorem}
        Let $\alpha > 3$ and $x :[s_{0}, +\infty[ \to \mathcal{H}$ be a solution trajectory of 
        \[
            \begin{cases}
                \ddot{x}(s) + \frac{\alpha}{s} \dot{x}(s) + \nabla f(x(s)) = 0, \\
                x(s_{0})=x_0, \quad \dot{x}(s_{0}) = x_{1}.
            \end{cases}
        \]
        Then, it holds
        \[
            f(x(s)) - {\inf\nolimits_{\mathcal{H}}f} = o\left( \frac{1}{s^{2}}\right) \quad \text{and} \quad \bigl\| \dot{x}(s)\bigr\| = o\left( \frac{1}{s}\right) 
      \qquad \mbox{as} \ s \to +\infty. \]
        Furthermore, $x(s)$ converges weakly to an element of $\argmin f$ as $s\to +\infty$. 
    \end{theorem}
    \begin{proof}
Let $\lambda >0$ and $s_0 > 0$. As per Proposition \ref{prop: Heavy Ball and AVD connected, function}, define $y(t) := x(\sigma(t))$. We know that $y : [t_{0}, +\infty) \to \mathcal{H}$ is a solution trajectory of \eqref{eq: heavy ball and avd connected, function case, heavy ball}. Taking into account the fact that $ \sigma\circ\tau : [s_{0}, +\infty) \to [s_{0}, +\infty)$ is the identity function, we have $x(s) = y(\tau(s))$, and therefore $\frac{\lambda s}{\alpha - 1} \dot{x}(s) = \dot{y}(\tau(s))$. Thus, according to Theorem \ref{thm:HBF}, we know that for $ b(t) = s_{0}^{2}\exp\left( \frac{2\lambda (t - t_{0})}{\alpha - 1}\right)$, we have 
            \[
                f(x(s)) - \inf\nolimits_{\mathcal{H}} f = o \left( \frac{1}{\int_{\frac{t_{0} + \tau(s)}{2}}^{\tau(s)}b(r) dr }\right) \:\:\: \text{and} \:\:\: \frac{\lambda s}{\alpha - 1}\bigl\| \dot{x}(s)\bigr\| = o\left( \sqrt{\frac{b(\tau(s))}{\int_{\frac{t_{0} + \tau(s)}{2}}^{\tau(s)} b(r)dr }}\right)\:\:\: \text{as} \:\:\: s\to +\infty.
            \]
Notice that for every $s \geq s_0$ it holds
            \begin{align*}
                \int_{\frac{t_{0} + \tau(s)}{2}}^{\tau(s)} b(r) dr &= s_{0}^{2}\int_{\frac{t_{0} + \tau(s)}{2}}^{\tau(s)} \exp\left( \frac{2\lambda (r - t_{0})}{\alpha - 1}\right) dr \\
                &= \frac{s_{0}^{2} (\alpha - 1)}{2\lambda} \left[ \exp\left( \frac{2\lambda (\tau(s) - t_{0})}{\alpha - 1}\right) - \exp\left( \frac{\lambda (\tau(s) - t_{0})}{\alpha - 1}\right)\right] \\
                &= \frac{s_{0}^{2} (\alpha - 1)}{2\lambda} \left[ \left(\frac{s}{s_{0}}\right)^{2} - \frac{s}{s_{0}}\right],
            \end{align*}
and $ b(\tau(s)) = s^{2}$. This gives $ f(x(s)) - \inf_{\mathcal{H}}f = o\left(\frac{1}{s^{2}}\right)$ and $\frac{\lambda s}{\alpha - 1} \bigl\| \dot{x}(s)\bigr\| \to 0$ as $s \to +\infty$, which verifies the rates we had claimed in the statement. Since $ y(t)$ converges weakly to a global minimizer of $f$ as $t\to +\infty$, so does $ x(s) = y(\tau(s))$ as $s\to +\infty$.  
\end{proof}

\section{Heavy Ball dynamics governed by a maximally monotone and continuous operator} \label{sec: Heavy ball system, operator}

In this section, we explore an analog to the Heavy Ball dynamics \eqref{HBF:eq}, but tailored to the solving of equations governed by maximal monotone operators. For a monotone and continuous operator $V : \mathcal{H} \to \mathcal{H}$, we consider the equation
\begin{equation}\label{eq: monotone equation}
\text{Find }y\in\mathcal{H}\text{ such that }V(y) = 0. 
\end{equation}
Attached to this equation, for $\lambda > 0$, $\mu : [t_{0}, +\infty[ \to \R_{++}$, a continuously differentiable and nondecreasing function, and $\gamma : [t_{0}, +\infty[ \to \R_{++}$ a continuous function, we study the asymptotic properties of the solutions to 
\begin{equation}\label{eq: heavy ball system operator}
\ddot{y}(t) + \lambda \dot{y}(t) + \mu(t) \frac{d}{dt}V(y(t)) + \gamma(t) V(y(t)) = 0 \quad \mbox{for} \ t\geq t_{0} > 0.
\end{equation}
When $V$ is not cocoercive, which is the case for example when it does not arise from a convex potential, the presence of the ``Hessian'' term $\frac{d}{dt}V(y(t))$ is necessary to obtain convergence rates for the residual and weak convergence results for the trajectory solution.

This section is organized as follows: in Subsection \ref{subsec: existence and uniqueness of solutions}, we address the existence and uniqueness of strong global solutions to our system. In Subsection \ref{subsec: convergence statements operator case}, we analyze the asymptotic properties of the global solutions to \eqref{eq: heavy ball system operator}. 

\subsection{Existence and uniqueness of {strong} global solutions}\label{subsec: existence and uniqueness of solutions}

We employ similar arguments to those used for the existence and uniqueness result for \eqref{HBF:eq}. First of all, we may rewrite \eqref{eq: heavy ball system operator} as a first order system in two variables. Indeed, it is straightforward to check that 
\begin{equation}\label{eq: second order heavy ball operator with initial conditions}
\begin{cases}
\ddot{y}(t) + \lambda \dot{y}(t) + \mu(t) \frac{d}{dt}V(y(t)) + \gamma(t) V(y(t)) = 0, \\
y(t_{0}) = y_{0}, \quad \dot{y}(t_{0}) = y_{1}
\end{cases}
\end{equation}
is equivalent to 
\begin{equation*}
\begin{cases}
\dot{y}(t) &= u(t) - \lambda y(t) - \mu(t) V(y(t)), \\
\dot{u}(t) &= \bigl(\dot{\mu}(t) - \gamma(t)\bigr) V(y(t)), \\
y(t_{0}) &= y_{0}, \\
u(t_{0}) &= \lambda y_{0} + y_{1} + \mu(t_{0}) V(y_{0}),  
\end{cases}
\end{equation*}
where 
\[
u(t) := \lambda y(t) + \dot{y}(t) + \mu(t) V(y(t)),
\]
or 
\begin{equation}\label{eq: operator system written as first order}
\begin{cases}
\bigl(\dot{y}(t), \dot{u}(t)\bigr) &= G(t, y(t), u(t)), \\
(y(t_{0}), u(t_{0})) &= \bigl(y_{0}, \: \lambda y_{0} + y_{1} + \mu(t_{0})V(y_{0})\bigr), 
\end{cases}
\end{equation}
where $G \colon [t_{0}, +\infty) \times \mathcal{H} \times \mathcal{H} \to \mathcal{H} \times \mathcal{H}$ is given by 
\[
G(t, y, u) := \Bigl( u - \lambda y - \mu(t) V(y), \: \bigl( \dot{\mu}(t) - \gamma(t)\bigr) V(y) \Bigr).
\]

The existence and uniqueness of global solutions to \eqref{eq: operator system written as first order} require strong assumptions, such as the Fréchet differentiability of $V$. It is more fitting to consider the existence and uniqueness of strong global solutions, which only require that $V$ is Lipschitz continuous. We call $y : [t_{0}, +\infty) \to \mathcal{H}$ a strong global solution of \eqref{eq: second order heavy ball operator with initial conditions} if the following statements hold:
\begin{enumerate}[left=0pt]
\item[(i)] $y, \dot{y} : [t_{0}, +\infty) \to \mathcal{H}$ are locally absolutely continuous, that is, absolutely continuous on each interval $[t_{0}, t_{1}]$;
\item[(ii)] $\ddot{y}(t) + \lambda \dot{y}(t) + \mu(t) \frac{d}{dt} V(y(t)) + \gamma(t) V(y(t)) = 0$ for almost every $t\in [t_{0}, +\infty)$;
\item[(iii)] $y(t_{0}) = y_{0}$ and $\dot{y}(t_{0}) = y_{1}$. 
\end{enumerate}
\begin{theorem}
Suppose further that  $V : \mathcal{H} \to \mathcal{H}$ is $L$-Lipschitz continuous. Then, for each $\left( y_{0}, y_{1} \right) \in \mathcal{H} \times \mathcal{H}$, \eqref{eq: second order heavy ball operator with initial conditions} has a unique strong global solution $y \colon [t_{0}, +\infty[ \to \mathcal{H}$.
\end{theorem}
\begin{proof}
We use the first order reformulation \eqref{eq: operator system written as first order}, and show the existence of strong global solutions to this  differential equation in the Hilbert space $\mathcal{H} \times \mathcal{H}$, which we endow with the inner product $\left\langle (y, u), (\overline{y},\overline{u}) \right\rangle_{\mathcal{H} \times \mathcal{H}} := \left\langle y , \overline{y} \right\rangle + \left\langle u , \overline{u} \right\rangle$ and corresponding norm $\| (y, u)\|_{\mathcal{H} \times \mathcal{H}} := \sqrt{\| y\|^{2} + \| u\|^{2}}$. 

\begin{enumerate}[(a)]
\item First, we check the Lipschitz continuity of $G(t, \cdot, \cdot)$ for each $t\in [t_{0}, +\infty)$.  For $(y, u), (\overline{y},\overline{u}) \in\mathcal{H} \times \mathcal{H}$,  we have
\begin{align*}
& \:\bigl\| G(t, y, u) - G(t, \overline{y}, \overline{u})\bigr\|_{\mathcal{H} \times\mathcal{H}} \\
\leq &\: \Bigl\| (u - \overline{u}) - \lambda (y - \overline{y}) - \mu(t) (V(y) - V(\overline{y}))\Bigr\| + \Bigl\| (\dot{\mu}(t) - \gamma(t)) (V(y) - V(\overline{y}))\Bigr\| \\
\leq &\: \| u - \overline{u}\| + \lambda \| y - \overline{y}\| + L| \mu(t)|\: \| y - \overline{y}\| + L\bigl| \dot{\mu}(t) - \gamma(t)\bigr| \|y - \overline{y}\| \\
\leq &\: \Bigl( 1 + \lambda + L| \mu(t)| + L\bigl| \dot{\mu}(t) - \gamma(t)\bigr| \Bigr) \bigl\| (y, u) - (\overline{y}, \overline{u})\bigr\|_{\mathcal{H} \times \mathcal{H}}.
\end{align*}
By hypothesis, the (dependent on $t$) Lipschitz constant attached to $\| (y, u) - (\overline{y}, \overline{u})\|_{\mathcal{H} \times \mathcal{H}}$ is locally integrable. 
\item Now, we prove the local integrability of $G(\cdot, y, u)$ for each $y, u\in \mathcal{H}$. We have 
\begin{align*}
\int_{t_{0}}^{t_{1}} \| G(s, y, u)\|_{\mathcal{H} \times \mathcal{H}} ds & \leq \int_{t_{0}}^{t_{1}} \Bigl(\Bigl\| u -\lambda y - \mu(s) V(y)\Bigr\| + \Bigl\| \bigl( \dot{\mu}(s) - \gamma(s)\bigr) V(y)\Bigr\|\Bigr) ds \\
&\leq \int_{t_{0}}^{t_{1}} \Bigl( \| u\| + \lambda \| y\| + | \mu(s)| \: \| V(y)\| + \left( \bigl| \dot{\mu}(s) \bigr| + \bigl| \gamma(s)\bigr| \right) \| V(y)\|\Bigr) ds,
\end{align*}
and by hypothesis the right-hand side is finite. 
\end{enumerate}

With these two conditions verified, we invoke the Cauchy-Lipschitz-Picard theorem (see, for example, \cite[Proposition 6.2.1]{Haraux}) to ensure the existence and uniqueness of a strong global solution to \eqref{eq: operator system written as first order}. By using the equivalence between the first- and second-order formulations, the statement follows. 
\end{proof}

\subsection{Convergence rates and weak convergence of the trajectories}\label{subsec: convergence statements operator case}

In this subsection,  we will assume that we have a global solution $y \colon [t_{0}, +\infty[ \to \mathcal{H}$ of \eqref{eq: heavy ball system operator}.

Before proceeding, we introduce the energy function that will help us in our analysis and show some estimates that we will need later. Suppose ${x_{*}}$ is a zero of $V$, and $t\mapsto y(t)$ solves \eqref{eq: heavy ball system operator} for $t\geq t_{0}$. For $0 < \eta < \lambda$, we define
\begin{align}
\mathcal{E}_{\eta}(t) := & \: \frac{1}{2} \Bigl\| 2\eta (y(t) - x_*) + 2\dot{y}(t) + \mu(t) V(y(t))\Bigr\|^{2} \tag{\ensuremath{\mathcal{E}_{\eta}^{1}(t)}} \hypertarget{eq: HBF monotone energy function line 1}{\label{eq: HBF monotone energy function line 1}}\\
&+ 2\eta (\lambda - \eta) \| y(t) - x_*\|^{2} \tag{\ensuremath{\mathcal{E}_{\eta}^{2}(t)}} \hypertarget{eq: HBF monotone energy function line 2}{\label{eq: HBF monotone energy function line 2}}\\
&+ 2\eta \mu(t) \langle y(t) - x_*, V(y(t))\rangle \tag{\ensuremath{\mathcal{E}_{\eta}^{3}(t)}} \hypertarget{eq: HBF monotone energy function line 3}{\label{eq: HBF monotone energy function line 3}}\\
&+ \frac{1}{2} \mu^{2}(t) \| V(y(t))\|^{2} \tag{\ensuremath{\mathcal{E}_{\eta}^{4}(t)}}. \hypertarget{eq: HBF monotone energy function line 4}{\label{eq: HBF monotone energy function line 4}}
\end{align}
We now compute the time derivative of $\mathcal{E}_{\eta}(\cdot)$ at a point $t\geq t_{0}$
\begin{align*}
\frac{d}{dt}\hyperlink{eq: HBF monotone energy function line 1}{\mathcal{E}_{\eta}^{1}(t)} = &\: \left\langle 2\eta (y(t) - x_*) + 2\dot{y}(t) + \mu(t) V(y(t)),\: 2\eta \dot{y}(t) + 2\ddot{y}(t) + \dot{\mu}(t) V(y(t)) + \mu(t) \frac{d}{dt} V(y(t)) \right\rangle \\
= &\: \bigg\langle 2\eta (y(t) - x_*) + 2\dot{y}(t) + \mu(t) V(y(t)), \\
&\quad \quad \quad 2(\eta - \lambda) \dot{y}(t) + \Bigl[ \dot{\mu}(t) - 2\gamma(t)\Bigr] V(y(t)) - \mu(t) \frac{d}{dt} V(y(t)) \bigg\rangle \\
= &\: 4\eta(\eta - \lambda) \bigl\langle y(t) - x_*, \dot{y}(t)\bigr\rangle + 2\eta \Bigl[ \dot{\mu}(t) - 2\gamma(t)\Bigr] \langle y(t) - x_*, V(y(t))\rangle\\
& - 2\eta\mu(t) \left\langle y(t) - x_*, \frac{d}{dt}V(y(t))\right\rangle + 4(\eta - \lambda) \bigl\| \dot{y}(t)\bigr\|^{2}\\
& + \Bigl\{ 2\Bigl[ \dot{\mu}(t) - 2\gamma(t)\Bigr] + 2 (\eta - \lambda)\mu(t)\Bigr\} \bigl\langle \dot{y}(t), V(y(t))\bigr\rangle - 2\mu(t) \left\langle \dot{y}(t), \frac{d}{dt}V(y(t))\right\rangle \\
&+ \mu(t) \Bigl[ \dot{\mu}(t) - 2\gamma(t)\Bigr] \| V(y(t))\|^{2} - \mu^{2}(t) \left\langle V(y(t)), \frac{d}{dt}V(y(t))\right\rangle, \\
\frac{d}{dt}\hyperlink{eq: HBF monotone energy function line 2}{\mathcal{E}_{\eta}^{2}(t)} = &\: 4 \eta (\lambda - \eta) \bigl\langle y(t) - x_*, \dot{y}(t)\bigr\rangle, \\
\frac{d}{dt}\hyperlink{eq: HBF monotone energy function line 3}{\mathcal{E}_{\eta}^{3}(t)} = &\: 2\eta \dot{\mu}(t) \langle y(t) - x_*, V(y(t))\rangle + 2\eta\mu(t) \bigl\langle \dot{y}(t), V(y(t))\bigr\rangle + 2\eta\mu(t) \left\langle y(t) - x_*, \frac{d}{dt}V(y(t))\right\rangle, \\
\frac{d}{dt}\hyperlink{eq: HBF monotone energy function line 4}{\mathcal{E}_{\eta}^{4}(t)} = &\: \mu(t) \dot{\mu}(t) \| V(y(t))\|^{2} + \mu^{2}(t) \left\langle V(y(t)), \frac{d}{dt}V(y(t))\right\rangle. 
\end{align*}
Putting everything together yields 
\begin{align}
\frac{d}{dt}\mathcal{E}_{\eta}(t) = &\:\frac{d}{dt}\hyperlink{eq: HBF monotone energy function line 1}{\mathcal{E}_{\eta}^{1}(t)} + \frac{d}{dt}\hyperlink{eq: HBF monotone energy function line 2}{\mathcal{E}_{\eta}^{2}(t)} + \frac{d}{dt}\hyperlink{eq: HBF monotone energy function line 3}{\mathcal{E}_{\eta}^{3}(t)} + \frac{d}{dt}\hyperlink{eq: HBF monotone energy function line 4}{\mathcal{E}_{\eta}^{4}(t)} \nonumber \\
= &\: \Bigl\{ 2\eta \Bigl[ \dot{\mu}(t) - 2\gamma(t)\Bigr] + 2\eta \dot{\mu}(t)\Bigr\}\langle y(t) - x_*, V(y(t))\rangle + 4(\eta - \lambda) \bigl\| \dot{y}(t)\bigr\|^{2} \nonumber \\
&+ \Bigl\{ 2\Bigl[ \dot{\mu}(t) - 2\gamma(t)\Bigr] + 2 (\eta - \lambda)\mu(t) + 2\eta\mu(t)\Bigr\} \bigl\langle \dot{y}(t), V(y(t))\bigr\rangle - 2\mu(t) \left\langle \dot{y}(t), \frac{d}{dt}V(y(t))\right\rangle \nonumber\\
&+ \Bigl\{ \mu(t) \Bigl[ \dot{\mu}(t) - 2\gamma(t)\Bigr] + \mu(t) \dot{\mu}(t)\Bigr\} \| V(y(t))\|^{2} \nonumber\\
= &\: 4\eta \Bigl[ \dot{\mu}(t) - \gamma(t)\Bigr]\langle y(t) - x_*, V(y(t))\rangle \label{eq: HBF monotone equation derivative line 1} \\
&+ 4(\eta - \lambda) \bigl\| \dot{y}(t)\bigr\|^{2} \label{eq: HBF monotone equation derivative line 2}\\
&+  2\Bigl\{ \Bigl[ 2(\eta - \lambda)\mu(t) + \lambda\mu(t) - 2\gamma(t)\Bigr] + \dot{\mu}(t)\Bigr\} \bigl\langle \dot{y}(t), V(y(t))\bigr\rangle \label{eq: HBF monotone equation derivative line 3}\\
& -2\mu(t) \left\langle \dot{y}(t), \frac{d}{dt}V(y(t))\right\rangle \label{eq: HBF monotone equation derivative line 4}\\
& + 2\mu(t) \Bigl[ \dot{\mu}(t) - \gamma(t)\Bigr] \| V(y(t))\|^{2}. \label{eq: HBF monotone equation derivative line 5}
\end{align}
Define 
\[
\varepsilon := \lambda - \eta > 0.
\]
Working with \eqref{eq: HBF monotone equation derivative line 2}, \eqref{eq: HBF monotone equation derivative line 3} and \eqref{eq: HBF monotone equation derivative line 5}, we will want to know the sign of 
\begin{equation}\label{eq: quad inequality HBF monotone}
-3 \varepsilon \bigl\| \dot{y}(t)\bigr\|^{2} + 2 \Bigl\{\Bigl[-2\varepsilon \mu(t) + \lambda \mu(t) - 2\gamma(t)\Bigr] + \dot{\mu(t)}\Bigr\} \bigl\langle \dot{y}(t), V(y(t))\bigr\rangle + \frac{4}{3}\mu(t) \Bigl[ \dot{\mu}(t) - \gamma(t)\Bigr] \| V(y(t))\|^{2}. 
\end{equation}
For that end, we make use of Lemma \ref{lem:quad} and set $X := \dot{y}(t)$, $Y := V(y(t))$ and $A$, $B$, $C$ chosen as follows: 
\[
A := -3\varepsilon, \quad B := \bigl( -2 \varepsilon \mu(t) + \lambda\mu(t) - 2\gamma(t)\bigr) + \dot{\mu}(t), \quad C := \frac{4}{3}\mu(t) \Bigl[ \dot{\mu}(t) - \gamma(t)\Bigr]. 
\]
We have 
\begin{align}
B^{2} - AC = &\:\Bigl[ \bigl(-2 \varepsilon \mu(t) + \lambda\mu(t) - 2\gamma(t)\bigr) + \dot{\mu}(t)\Bigr]^{2} + 3\cdot \frac{4}{3} \varepsilon \mu(t) \Bigl[ \dot{\mu}(t) - \gamma(t)\Bigr] \nonumber\\
= &\:\Bigl[ -2\varepsilon \mu(t) + \lambda\mu(t) - 2\gamma(t)\Bigr]^{2} - 4\varepsilon \mu(t)\dot{\mu}(t) + 2 \bigl(\lambda\mu(t) - 2\gamma(t)\bigr) \dot{\mu}(t) + \bigl(\dot{\mu}(t)\bigr)^{2} \nonumber\\
&+ 4\varepsilon \mu(t) \dot{\mu}(t) - 4\varepsilon \mu(t) \gamma(t) \nonumber\\
= &\: 4\varepsilon^{2} \mu^{2}(t) - 4\varepsilon \mu(t) \bigl( \lambda \mu(t) - 2\gamma(t)\bigr) + \bigl( \lambda\mu(t) - 2\gamma(t)\bigr)^{2} \! + 2\bigl(\lambda\mu(t) - 2\gamma(t)\bigr)\dot{\mu}(t) \nonumber \\ 
& \ + \bigl(\dot{\mu}(t)\bigr)^{2} - 4\varepsilon\mu(t)\gamma(t) \nonumber\\
= &\: \mu^{2}(t) \left\{ 4\varepsilon^{2} - 4\left(\lambda - \frac{\gamma(t)}{\mu(t)}\right)\varepsilon + \left[ \lambda - \frac{2\gamma(t)}{\mu(t)} + \frac{\dot{\mu}(t)}{\mu(t)}\right]^{2}\right\}. \label{eq: parabola}
\end{align}

\color{black}

We will be working under an assumption which will allow us to show that the parabola \eqref{eq: parabola} eventually becomes negative for every $\varepsilon$ in a certain $I \subseteq ]0, \lambda[$. \vspace{1ex}

\begin{mdframed}
Suppose that $\lambda > 0$, $\mu : [t_{0}, +\infty[ \to \R_{++}$ and $\gamma : [t_{0}, +\infty[ \to \R_{++}$ satisfy further
\begin{equation}
\label{HBF:b-lambda monotone}
\lim_{t\to +\infty} \frac{\gamma(t)}{\mu(t)} =: L > 0, \quad \sup_{t\geq t_{0}} \frac{\dot{\mu}(t)}{\gamma(t)} < 1 \quad \text{and} \quad 2\lambda - 3L + \inf_{t\geq t_{0}} \frac{\dot{\mu}(t)}{\mu(t)} > 0. 
\end{equation}
\end{mdframed}
\begin{remark}
Assumption \eqref{HBF:b-lambda monotone} implies in particular that $\lambda > L$, a fact that we will need later. Let us show this claim here. Since
\[
2\lambda > 3L - \inf_{t \geq t_{0}}\frac{\dot{\mu}(t)}{\mu(t)}, \ \text{there exists }\delta > 0\text{ such that} \ 2\lambda - \delta \geq 3L - \inf_{t\geq t_{0}}\frac{\dot{\mu}(t)}{\mu(t)}.
\]
Since $\lim_{t\to +\infty} \frac{\gamma(t)}{\mu(t)} = L$, for every  given $0 < \varepsilon < \delta$ there exists $t_{\varepsilon} \geq t_{0}$ such that for every $t\geq t_{\varepsilon}$
\[
L - \varepsilon \leq \frac{\gamma(t)}{\mu(t)} \leq L + \varepsilon, \ \text{and in particular,} \ \frac{\dot{\mu}(t)}{\mu(t)} \leq (L + \varepsilon) \frac{\dot{\mu}(t)}{\gamma(t)} \leq L + \varepsilon, 
\]
where in the last estimate we used the assumption $\sup_{t\geq t_{0}}\frac{\dot{\mu}(t)}{\gamma(t)} < 1$ and the nonincreasing property of $\mu$. Now, for every $t\geq t_{\varepsilon}$ we have 
\[
2\lambda - \delta \geq 3L - \inf_{t \geq t_{0}}\frac{\dot{\mu}(t)}{\mu(t)} \geq 3L - \frac{\dot{\mu}(t)}{\mu(t)} \geq 3L - L - \varepsilon = 2L - \varepsilon, 
\]
from which we deduce
\[
2\lambda > 2L + \delta - \varepsilon > 2L \ \text{and thus} \  \lambda > L. 
\]
\end{remark}

\begin{proposition}\label{prop:HBF-like}
Suppose further that $\lambda$, $\mu(\cdot)$ and $\gamma(\cdot)$ satisfy assumption \eqref{HBF:b-lambda monotone}. Let $x_*$ be a zero of $V$ and $y: [t_0,+\infty [ \rightarrow \cH$ be a solution trajectory of \eqref{eq: heavy ball system operator}. Then, the following statements are true:
\begin{enumerate}
\item \label{prop:HBF-like:lim} \emph{(energy functions convergence)}
The limit $\lim_{t \to + \infty} \cE_{\lambda - \varepsilon} \left( t \right) \in \R$ exists for every $\varepsilon$ in a given interval I $\subseteq ] 0 , \lambda [$.

\item \label{prop:HBF-like:int} \emph{(integrability results)}
It holds
\begin{align*}
\int_{{t_{0}}}^{+\infty} \mu(t) \langle y(t) - x_*, V(y(t))\rangle dt < +\infty, 
\int_{{t_{0}}}^{+\infty} \bigl\| \dot{y}(t)\bigr\|^{2} dt < +\infty, 
\int_{{t_{0}}}^{+\infty} \mu^{2}(t) \|V(y(t))\|^{2} dt < +\infty. 
\end{align*}
\end{enumerate}
\end{proposition}
\color{black}
\begin{proof}
\begin{enumerate}
\item 
We wish to choose $\varepsilon_{1}$ and $\varepsilon_{2}$, independent of time, such that the parabola \eqref{eq: parabola} becomes negative for large enough $t$. According to \eqref{HBF:b-lambda monotone}, we can choose $\delta_{1}, \delta_{2}$ such that 
\begin{equation}\label{eq: strict supremum}
\max \left\lbrace 2 - \frac{\lambda}{L} , \sup_{t\geq t_{0}} \frac{\dot{\mu}(t)}{\gamma(t)} \right\rbrace < \delta_{1} < \delta_{2} < 1 . 
\end{equation}
After some algebraic manipulation, we deduce that
\[
\lambda - \frac{2\gamma(t)}{\mu(t)} + \frac{\dot{\mu}(t)}{\mu(t)} < \lambda + (\delta_{1} - 2) \frac{\gamma(t)}{\mu(t)} < \lambda + (\delta_{2} - 2)\frac{\gamma(t)}{\mu(t)} < \lambda - \frac{\gamma(t)}{\mu(t)} \quad \forall t\geq t_{0}.
\]
Moreover, we know that there exist $\Tilde{\delta}_{2} > \Tilde{\delta}_{1} > 0$ satisfying
\begin{equation}
\label{eq: discriminant away from zero 1:a} 
0 < \lambda + (\delta_{1} - 2)L < \Tilde{\delta}_{1} < \Tilde{\delta}_{2} < \lambda + (\delta_{2} - 2)L .
\end{equation}
Using $\lim_{t\to +\infty} \frac{\gamma(t)}{\mu(t)} = L > 0$, we know that there exist $t_{1}\geq t_{0}$ such that
\begin{equation}
\label{eq: discriminant away from zero 1:b} 
\lambda - \frac{2\gamma(t)}{\mu(t)} + \frac{\dot{\mu}(t)}{\mu(t)} < \lambda + (\delta_{1} - 2) \frac{\gamma(t)}{\mu(t)} \leq \Tilde{\delta}_{1} < \Tilde{\delta}_{2} \leq \lambda + (\delta_{2} - 2)\frac{\gamma(t)}{\mu(t)} < \lambda - \frac{\gamma(t)}{\mu(t)} \quad \forall t\geq t_{1}. 
\end{equation}
On the other hand, according to \eqref{HBF:b-lambda monotone}, we know that 
\[
L - \lambda < \lambda - 2L + \inf_{t \geq t_{0}}\frac{\dot{\mu}(t)}{\mu(t)}.
\]
Similarly, we can construct $\Tilde{\delta}_{4} > \Tilde{\delta}_{3} > 0$ and $t_{2} \geq t_{1}$ such that  
\begin{gather}
L - \lambda < -\Tilde{\delta}_{4} < - \Tilde{\delta}_{3} < \lambda - 2L + \inf_{t \geq t_{0}}\frac{\dot{\mu}(t)}{\mu(t)} \quad \text{and} \nonumber\\
\frac{\gamma(t)}{\mu(t)} - \lambda \leq -\Tilde{\delta}_{4} < - \Tilde{\delta}_{3} \leq \lambda - \frac{2\gamma(t)}{\mu(t)} + \inf_{t \geq t_{0}}\frac{\dot{\mu}(t)}{\mu(t)} \leq \lambda - \frac{2\gamma(t)}{\mu(t)} + \frac{\dot{\mu}(t)}{\mu(t)} \quad \forall t\geq t_{2}. \label{eq: discriminant away from zero 2}
\end{gather}
Combining \eqref{eq: discriminant away from zero 1:b} and \eqref{eq: discriminant away from zero 2} yields 
\[
\frac{\gamma(t)}{\mu(t)} - \lambda \leq -\Tilde{\delta}_{4} < -\Tilde{\delta}_{3} \leq \lambda - \frac{2\gamma(t)}{\mu(t)} + \frac{\dot{\mu}(t)}{\mu(t)} \leq \Tilde{\delta}_{1} < \Tilde{\delta}_{2} \leq \lambda - \frac{\gamma(t)}{\mu(t)} \quad \forall t\geq t_{2}. 
\]
Without loss of generality, since both cases can be handed identically, we may assume that 
\[
\max \left\lbrace \Tilde{\delta}_{1}, \Tilde{\delta}_{3} \right\rbrace = \Tilde{\delta}_{1}.
\]
Therefore, the reduced discriminant of \eqref{eq: parabola} satisfies 
\begin{equation}\label{eq: Delta_t stays away from zero}
\Delta_{t} := 4 \left[ \left(\lambda - \frac{\gamma(t)}{\mu(t)}\right)^{2} - \left(\lambda - \frac{2\gamma(t)}{\mu(t)} + \frac{\dot{\mu}(t)}{\mu(t)}\right)^{2} \right] \geq 4 \left( \Tilde{\delta}_{2}^{2} - \Tilde{\delta}_{1}^{2} \right) > 0 \quad \forall t\geq t_{2}. 
\end{equation}
Now, choose $\delta_{3} > 0$ such that 
\[
\delta_{3} < \min \left\{ L, \: \lambda - L, \: \dfrac{1}{2} \sqrt{\Tilde{\delta}_{2}^{2} - \Tilde{\delta}_{1}^{2}} \right\}.
\]
Once again using the fact that $\lim_{t\to +\infty} \frac{\gamma(t)}{\mu(t)} = L$, there exists $t_{3} \geq t_{2}$ such that 
\begin{equation}\label{eq: gamma/mu stays close to L}
\lambda - L - \delta_{3} \leq \lambda - \frac{\gamma(t)}{\mu(t)} \leq \lambda - L + \delta_{3} \quad \forall t\geq t_{3}. 
\end{equation}
Going back to \eqref{eq: parabola}, the roots of the parabola are given, for each $t$, by 
\[
\underline{\varepsilon}_{t} := \frac{1}{2}\left( \lambda - \frac{\gamma(t)}{\mu(t)}\right) - \dfrac{1}{4} \sqrt{\Delta_{t}} \quad \text{and} \quad \overline{\varepsilon}_{t} := \frac{1}{2}\left( \lambda - \frac{\gamma(t)}{\mu(t)}\right) + \dfrac{1}{4} \sqrt{\Delta_{t}}. 
\]
Now, using \eqref{eq: Delta_t stays away from zero} and \eqref{eq: gamma/mu stays close to L} we can deduce
\begin{gather*}
\delta_{3} < \dfrac{1}{2} \sqrt{\Tilde{\delta}_{2}^{2} - \Tilde{\delta}_{1}^{2}} \quad \Rightarrow \quad \lambda - L + \delta_{3} - \sqrt{\Tilde{\delta}_{2}^{2} - \Tilde{\delta}_{1}^{2}} < \lambda - L - \delta_{3} \\
\Rightarrow \quad \underline{\varepsilon}_{t} = \frac{1}{2}\left(\lambda - \frac{\gamma(t)}{\mu(t)}\right) - \dfrac{1}{4} \sqrt{\Delta_{t}}  < \frac{1}{2} \left( \lambda - L + \delta_{3} - \sqrt{\Tilde{\delta}_{2}^{2} - \Tilde{\delta}_{1}^{2}} \right) < \dfrac{1}{2} \left( \lambda - L - \delta_{3} \right) \quad \forall t\geq t_{3}.
\end{gather*}
Similarly, we obtain
\[
\frac{1}{2} \left( \lambda - L + \delta_{3} \right) < \overline{\varepsilon}_{t} \quad \forall t\geq t_{3}.
\]
In other words, we have
\[
I := \left[ \frac{1}{2} \left( \lambda - L - \delta_{3} \right), \frac{1}{2} \left( \lambda - L + \delta_{3} \right) \right] \subseteq \bigl] \underline{\varepsilon}_{t}, \overline{\varepsilon}_{t}\bigr[ \quad \forall t\geq t_{3}. 
\]
The way we chose $\delta_{3}$ also ensures that 
\[
I \subseteq \: ]0, \lambda[.
\]
Thus, for every $\varepsilon\in I$, the term \eqref{eq: parabola} becomes negative, i.e., 
\[
B^{2} - AC < 0 \quad \text{for every} \ \varepsilon \in I \ \text{and every} \ t\geq t_{3}. 
\]
Having shown \eqref{eq: quad inequality HBF monotone}, we go back to \eqref{eq: HBF monotone equation derivative line 1}-\eqref{eq: HBF monotone equation derivative line 5}. For every $\eta = \lambda - \varepsilon > 0$ with $\varepsilon \in I$, it holds
\begin{align}
\frac{d}{dt}\mathcal{E}_{\eta}(t) \leq &\: 4\eta \Bigl[ \dot{\mu}(t) - \gamma(t)\Bigr] \langle y(t) - x_*, V(y(t))\rangle \nonumber\\
&+ (\eta - \lambda) \bigl\| \dot{y}(t)\bigr\|^{2} \nonumber\\
&+ \frac{2}{3}\mu(t) \Bigl[ \dot{\mu}(t) - \gamma(t)\Bigr] \| V(y(t))\|^{2} \nonumber \\
\leq &\: 0 \quad \forall t\geq t_{3}. \label{eq: derivative of energy function after quad HBF monotone} 
\end{align}
This means that for every $\eta = \lambda - \varepsilon$, where $\varepsilon \in I$, $\mathcal{E}_{\eta}(\cdot)$ is nonincreasing on $[{t_{3}}, +\infty[$, therefore
\begin{gather}
0 \leq \mathcal{E}_{\eta}(t) \leq \mathcal{E}_{\eta}(t_{3}) \quad \text{for } t\geq t_{3}  \label{eq: HBF monotone energy functions are bounded}\\
\intertext{and}
\lim_{t\to +\infty} \mathcal{E}_{\eta}(t) \in \R \text{ exists.} \label{eq: HBF monotone energy function has a limit}
\end{gather}
In particular, going back to \eqref{eq: HBF monotone energy function line 1}-\eqref{eq: HBF monotone energy function line 4}, we deduce that
\begin{equation}
\| V(y(t))\| \leq \frac{\sqrt{2\mathcal{E}_{\eta}(t_{3})}}{\mu(t)} \quad \mbox{and} \quad \langle y(t) - x_*, V(y(t))\rangle \leq \frac{\mathcal{E}_{\eta}(t_{3})}{2\eta \mu(t)} \quad \forall t\geq t_{3}.
\end{equation}

\item 
From \eqref{eq: strict supremum} and \eqref{eq: gamma/mu stays close to L}, for every $t\geq t_{3}$ we have
\begin{gather*}
\frac{\dot{\mu}(t)}{\gamma(t)} \leq \delta_{1} < 1 \quad \text{and} \quad L - \delta_{3} \leq \frac{\gamma(t)}{\mu(t)} \leq L + \delta_{3} \quad \Rightarrow \quad (1 - \delta_{1}) \gamma(t) \leq \gamma(t) - \dot{\mu}(t) \\
\Rightarrow \quad (1 - \delta_{1}) (L - \delta_{3}) \mu(t) \leq (1 - \delta_{1}) \gamma(t) \leq \gamma(t) - \dot{\mu}(t).
\end{gather*}
Going back to \eqref{eq: derivative of energy function after quad HBF monotone}, we integrate the inequality from $t_{3}$ to $t\geq t_{3}$ and obtain 
\begin{align*}
&\: 4\eta (L - \delta_{3})(1 - \delta_{1})\int_{t_{3}}^{t} \mu(s) \langle y(s) - x_*, V(y(s))\rangle ds + (\lambda - \eta) \int_{t_{3}}^{t} \bigl\| \dot{y}(s)\bigr\|^{2} ds \\
&+ \frac{2}{3} (L - \delta_{3})(1 - \delta_{1}) \int_{t_{3}}^{t} \mu^{2}(s) \| V(y(s))\|^{2} ds \\ 
\leq &\: 4\eta \int_{t_{3}}^{t} \Bigl[ \gamma(s) - \dot{\mu}(s)\Bigr] \langle y(s) - x_*, V(y(s))\rangle ds + (\lambda - \eta) \int_{t_{3}}^{t} \bigl\| \dot{y}(s)\bigr\|^{2} ds\\
& \: + \frac{2}{3} \int_{t_{3}}^{t} \mu(s) \Bigl[ \gamma(s) - \dot{\mu}(s)\Bigr] \| V(y(s))\|^{2} ds \\
\leq &\: \mathcal{E}_{\eta}(t_{3}) - \mathcal{E}_{\eta}(t), 
\end{align*}
which produces
\begin{align}
\int_{t_{3}}^{+\infty} \mu(t) \langle y(t) - x_*, V(y(t))\rangle dt &< +\infty, \label{eq: HBF monotone integrability result 1}\\
\int_{t_{3}}^{+\infty} \bigl\| \dot{y}(t)\bigr\|^{2} dt &< +\infty, \label{eq: HBF monotone integrability result 2}\\
\int_{t_{3}}^{+\infty} \mu^{2}(t) \|V(y(t))\|^{2} dt &< +\infty. \label{eq: HBF monotone integrability result 3} 
\end{align}
\end{enumerate}
\end{proof}
\color{black}

We are now in conditions to state and prove the main theorem of this section. 
\begin{theorem}\label{thm: heavy ball operator}
Suppose that $\lambda$, $\mu(\cdot)$ and $\gamma(\cdot)$ satisfy assumption \eqref{HBF:b-lambda monotone}. Let $x_*$ be a zero of $V$ and $y: [t_0,+\infty [ \rightarrow \cH$ be a solution trajectory of \eqref{eq: heavy ball system operator}. Then, the following statements are true:
\begin{enumerate}
\item[(i)] (convergence rates) It holds 
\begin{equation}
\| V(y(t))\| = o \left( \frac{1}{\mu(t)}\right), \quad \langle y(t) - x_*, V(y(t))\rangle = o\left( \frac{1}{\mu(t)}\right), \quad \bigl\| \dot{y}(t)\bigr\| \to 0 \quad \mbox{as} \ t \rightarrow +\infty.
\end{equation}
\item[(ii)] (trajectory convergence) The solution trajectory $y(t)$ converges weakly to a zero of $V$ as $t\to +\infty$. 
\end{enumerate}

\end{theorem}

\begin{proof}
\begin{enumerate}
\item 
Let $ I \subseteq ]0, \lambda[$ be the interval provided by Proposition \ref{prop:HBF-like} \ref{prop:HBF-like:lim}.  Choose $\varepsilon_{1}, \varepsilon_{2} \in I$ such that $\varepsilon_{1} \neq \varepsilon_{2}$. Set $\eta_{i} := \lambda - \varepsilon_{i}$, $i = 1, 2$. We have
\begin{align*}
\mathcal{E}_{\eta_{2}}(t) - \mathcal{E}_{\eta_{1}}(t) = & \ \ \ \frac{1}{2} \Bigl\| 2\eta_{2}(y(t) - x_*) + 2\dot{y}(t) + \mu(t) V(y(t))\Bigr\|^{2}\\
& \! - \frac{1}{2}\Bigl\| 2\eta_{1}(y(t) - x_*) + 2\dot{y}(t) + \mu(t) V(y(t))\Bigr\|^{2} \\
&+ 2 \Bigl[ \eta_{2}(\lambda - \eta_{2}) - \eta_{1}(\lambda - \eta_{1})\Bigr] \| y(t) - x_*\|^{2} + 2 (\eta_{2} - \eta_{1})\mu(t) \langle y(t) - x_*, V(y(t))\rangle \\
= &\ \ \ 2 \bigl( \eta_{2}^{2} - \eta_{1}^{2}\bigr) \| y(t) - x_*\|^{2} + 2(\eta_{2} - \eta_{1}) \Bigl\langle y(t) - x_*, 2\dot{y}(t) + \mu(t) V(y(t))\Bigr\rangle \\
&+ \Bigl[ 2\lambda(\eta_{2} - \eta_{1}) - 2\bigl(\eta_{2}^{2} - \eta_{1}^{2}\bigr)\Bigr]\| y(t) - x_*\|^{2} + 2(\eta_{2} - \eta_{1}) \mu(t) \langle y(t) - x_*, V(y(t))\rangle \\
= &\ \ \ 4 (\eta_{2} - \eta_{1}) \left[ \frac{\lambda}{2} \| y(t) - x_*\|^{2} + \Bigl\langle y(t) - x_*, \dot{y}(t) + \mu(t) V(y(t))\Bigr\rangle\right].
\end{align*}
Define, for $t\geq t_{0}$, 
\[
p(t) := \frac{\lambda}{2} \| y(t) - x_*\|^{2} + \Bigl\langle y(t) - x_*, \dot{y}(t) + \mu(t) V(y(t))\Bigr\rangle.
\]
Because of \eqref{eq: HBF monotone energy function has a limit}, and since $\eta_{2} - \eta_{1} \neq 0$, we deduce that 
\[
\lim_{t\to +\infty}p(t) \in \R \text{ exists.}
\]
With this at hand, we can now rewrite $\mathcal{E}_{\eta_{1}}(t)$ for every $t \geq t_0$ as 
\begin{align*}
\mathcal{E}_{\eta_{1}}(t) = &\: \frac{1}{2}\Bigl\| 2\eta_{1}(y(t) - x_*) + 2\dot{y}(t) + \mu(t) V(y(t))\Bigr\|^{2} + 2\eta_{1} (\lambda - \eta_{1}) \| y(t) - x_*\|^{2} \\
&+ 2\eta_{1} \mu(t) \langle y(t) - x_*, V(y(t))\rangle + \frac{1}{2} \mu^{2}(t) \| V(y(t))\|^{2} \\
= &\: \frac{1}{2} \left[ 4\eta_{1}^{2} \| y(t) - x_*\|^{2} + 2\eta_{1} \Bigl\langle y(t) - x_*, 2\dot{y}(t) + \mu(t) V(y(t))\Bigr\rangle + \Bigl\| 2\dot{y}(t) + \mu(t) V(y(t))\Bigr\|^{2}\right] \\
&+ 2\eta_{1}\lambda \| y(t) - x_*\|^{2} - 2\eta_{1}^{2} \| y(t) - x_*\|^{2} + 2\eta_{1} \mu(t) \langle y(t) - x_*, V(y(t))\rangle + \frac{1}{2} \mu^{2}(t) \| V(y(t))\|^{2} \\
= &\: 4\eta_{1} p(t) + \frac{1}{2} \Bigl\| 2\dot{y}(t) + \theta b(t) V(y(t))\Bigr\|^{2} +  \frac{1}{2} \mu^{2}(t) \| V(y(t))\|^{2} \\
= &\: 4\eta_{1} p(t) + \Bigl\| \dot{y}(t) + \mu(t) V(y(t))\Bigr\|^{2} + \bigl\| \dot{y}(t)\bigr\|^{2}.
\end{align*}
Define 
\[
h(t) := \Bigl\| \dot{y}(t) + \mu(t) V(y(t))\Bigr\|^{2} + \bigl\| \dot{y}(t)\bigr\|^{2}.
\]
Since both $\lim_{t\to +\infty}\mathcal{E}_{\eta_{1}}(t)$ and $\lim_{t\to +\infty} p(t)$ exist, so does $\lim_{t\to +\infty} h(t)$. Notice that
\[
\int_{t_{0}}^{t} h(s) ds \leq 3\int_{t_{0}}^{t} \bigl\| \dot{y}(s)\bigr\|^{2} ds + 2 \int_{t_{0}}^{t} \mu^{2}(s)\| V(y(s))\|^{2} ds, 
\]
and the integrals on the right-hand side remain finite as $t\to +\infty$ according to Proposition \ref{prop:HBF-like} \ref{prop:HBF-like:lim} \ref{prop:HBF-like:int}, i.e., $\int_{t_{0}}^{+\infty} h(t) dt < +\infty$. Combining this with the fact that $\lim_{t\to +\infty} h(t)$ exists allows us to deduce 
\[
\lim_{t\to +\infty} h(t) = 0. 
\]
This implies 
\[
\lim_{t\to +\infty} \|\dot{y}(t) \| = 0, 
\]
combining this with $\lim_{t\to +\infty} \bigl\| \dot{y}(t) + \mu(t) V(y(t))\bigr\| = 0$ gives
\[
\lim_{t\to +\infty} \mu(t) \| V(y(t))\| = 0, \quad \text{or} \quad \| V(y(t))\| = o\left( \frac{1}{\mu(t)}\right) \quad \text{as} \quad t\to +\infty.
\]
Using the boundedness of $t \mapsto \| y(t) - x_*\|$ produces 
\[
\langle y(t) - x_*, V(y(t))\rangle = o \left(\frac{1}{\mu(t)}\right) \quad \text{as} \quad t\to +\infty.
\]
\item 
We will make use of Opial's Lemma. Define, for $t\geq t_{0}$, 
\[
q(t) := \frac{1}{2}\| y(t) - x_*\|^{2}  + \int_{t_{0}}^{t} \mu(s) \langle y(s) - x_*, V(y(s))\rangle ds.
\]
Recalling the definition of $p(\cdot)$, we have 
\begin{align*}
\lambda q(t) + \dot{q}(t) &= \frac{\lambda}{2} \| y(t) - x_*\|^{2} + \Bigl\langle y(t) - x_*, \dot{y}(t) + \mu(t) V(y(t))\Bigr\rangle + \lambda \int_{t_{0}}^{t} \mu(s) \langle y(s) - x_*, V(y(s))\rangle ds \\
&= p(t) + \lambda \int_{t_{0}}^{t} \mu(s) \langle y(s) - x_*, V(y(s))\rangle ds.
\end{align*}
According to \eqref{eq: HBF monotone integrability result 1}, the integral in the previous sum converges as $t\to +\infty$. Since we already established that $\lim_{t\to +\infty} p(t)$ exists, we are lead to 
\[
\lim_{t\to +\infty}\bigl( \lambda q(t) + \dot{q}(t)\bigr) \in \R \text{ exists.}
\]
Using Lemma \ref{lem:q}, we obtain the existence of $\lim_{t\to +\infty}q(t)$. Using again that $\int_{t_{0}}^{+\infty} \mu(s) \langle y(s) - x_*, V(y(s))\rangle ds < +\infty$, we finally deduce that 
\[
\lim_{t\to +\infty} \| y(t) - x_*\| \text{ exists.}
\]
Thus, the first condition of Opial's Lemma is met. For the second condition, let $\overline{y}$ be a sequential cluster point $y(t)$ as $t\to +\infty$, which means there exists a sequence $({t_{n}})_{n\in \N} \subseteq [t_{0}, +\infty[$ such that ${t_{n}}\to +\infty$  as $n\to +\infty$ and 
\[
y({t_{n}}) \rightharpoonup \overline{y} \quad \text{as} \quad n\to +\infty. 
\]
Since we already know that $\| V(y(t))\| = o\left(\frac{1}{\mu(t)}\right)$ as $t\to +\infty$ and $\mu(\cdot)$ is nondecreasing, we have $V(y({t_{n}})) \to 0$ as $n\to +\infty$. Since $V$ is maximally monotone, its graph is closed in $\mathcal{H}^{\text{weak}} \times \mathcal{H}^{\text{strong}}$, thus
\[
V(\overline{y}) = 0. 
\]
Now both conditions of Opial's Lemma (see Lemma \ref{Opial}) are fulfilled, from which we finally conclude the proof of this theorem.
\end{enumerate}
\end{proof}

\section{Connection with a system with asymptotic vanishing damping governed by a monotone and continuous operator}\label{sec: connection between heavy ball and AVD, operator case}

Here, as it was done before, we present two systems attached to \eqref{eq: monotone equation}: one has a fixed viscosity parameter, i.e., the Heavy Ball dynamics \eqref{eq: heavy ball system operator}, while the other features an asymptotically vanishing viscosity coefficient accompanying the velocity; precisely speaking, these are exactly the Fast OGDA dynamics (see \cite{fOGDA}) for the case $\beta(\cdot) \equiv 1$. While at first glance they might appear fundamentally different, we will see that one is a time-rescaled version of the other.
\subsection{Two equivalent dynamical systems through time rescaling}
Similar to what was done in Section \ref{sec: connection between heavy ball and AVD, function case}, we start with a trajectory solution $y \colon [t_{0}, +\infty[ \to \mathcal{H}$ of
\begin{equation}\label{eq: heavy ball system operator 2}
\ddot{y}(t) + \lambda \dot{y}(t) + \mu(t) \frac{d}{dt}V(y(t)) + \gamma(t) V(y(t)) = 0, 
\end{equation}
and define $x(s) := y(\tau(s))$, where $\tau \colon [s_{0}, +\infty[ \to [t_{0}, +\infty[$ is a continuously differentiable function such that $\dot{\tau}(s) > 0$ for every $s\geq s_{0} > 0$ and $\lim_{s \to +\infty} \tau(s) = +\infty$. We have 
\begin{align*}
\dot{x}(s)  = \dot{\tau}(s) \dot{y}(\tau(s)) 
\quad \textrm{ and } \quad
\ddot{x}(s) = \ddot{\tau}(s) \dot{y}(\tau(s)) + \bigl(\dot{\tau}(s)\bigr)^{2} \ddot{y}(\tau(s)) .
\end{align*}
These expressions lead to
\begin{equation*}
\dot{y}(\tau(s)) = \frac{1}{\dot{\tau}(s)} \dot{x}(s)
\quad \textrm{ and } \quad
\ddot{y}(\tau(s)) = \frac{1}{\bigl(\dot{\tau}(s)\bigr)^{2}} \Bigl[ \ddot{x}(s) - \ddot{\tau}(s) \dot{y}(\tau(s))\Bigr] = \frac{1}{\bigl(\dot{\tau}(s)\bigr)^{2}} \left[ \ddot{x}(s) - \frac{\ddot{\tau}(s)}{\dot{\tau}(s)} \dot{x}(s)\right] .
\end{equation*}
Moreover
\begin{gather*}
\frac{d}{ds} \Bigl( V(y(\cdot)) \circ \tau\Bigr)(s) = \dot{\tau}(s) \frac{d}{dt} V(y(t)) \Bigr|_{t = \tau(s)} \: \Rightarrow \: \frac{d}{dt} V(y(t)) \Bigr|_{t = \tau(s)} = \frac{1}{\dot{\tau}(s)} \frac{d}{ds} V(x(s)).
\end{gather*}
Now, plugging $t = \tau(s)$ in \eqref{eq: heavy ball system operator 2} gives
\begin{equation*}
\frac{1}{\bigl(\dot{\tau}(s)\bigr)^{2}} \left[ \ddot{x}(s) - \frac{\ddot{\tau}(s)}{\dot{\tau}(s)}\dot{x}(s)\right] + \frac{\lambda}{\dot{\tau}(s)}\dot{x}(s) +  \mu(\tau(s)) \frac{1}{\dot{\tau}(s)} \frac{d}{ds}V(x(s)) + \gamma(\tau(s)) V(x(s)) = 0
\end{equation*}
or, equivalently,
\begin{equation}\label{eq: time rescaled heavy ball, operator}
\ddot{x}(s) + \left[ \lambda \dot{\tau}(s) - \frac{\ddot{\tau}(s)}{\dot{\tau}(s)}\right] \dot{x}(s) + \dot{\tau}(s) \mu(\tau(s)) \frac{d}{ds} V(x(s)) + \big( \dot{\tau}(s)\bigr)^{2} \gamma(\tau(s)) V(x(s)) = 0.
\end{equation}
We briefly recall that the Fast OGDA dynamics \cite{fOGDA}, for constant $\beta(\cdot) \equiv 1$, read like
\begin{equation}\label{eq: fOGDA}
\ddot{x}(s) + \frac{\alpha}{s} \dot{x}(s) + \frac{d}{ds} V(x(s)) + \frac{\alpha}{2s} V(x(s)) = 0 \quad \mbox{for} \ s \geq s_0  >0.
\end{equation}
Going back to \eqref{eq: time rescaled heavy ball, operator}, if we want an asymptotic vanishing viscosity coefficient accompanying the velocity, we need to ask  for
\[
\begin{cases}
\displaystyle\lambda \dot{\tau}(s) - \frac{\ddot{\tau}(s)}{\dot{\tau}(s)} &= \displaystyle\frac{\alpha}{s} \\
\displaystyle\tau(s_{0}) &= t_{0}.
\end{cases}  
\]
Since it was done before in Section \ref{sec: connection between heavy ball and AVD, function case}, we don't repeat the derivation for finding the solution here. We know that for $s\geq s_{0} > 0$, the function 
\begin{equation}
\label{defi:tau}
\tau(s) := \frac{\alpha - 1}{\lambda} \ln(s) + \left(-\frac{\alpha - 1}{\lambda}\ln(s_{0}) + t_{0}\right) = \frac{\alpha - 1}{\lambda}\ln\left(\frac{s}{s_{0}}\right) + t_{0}
\end{equation}
satisfies this differential equation. We have 
\[
\dot{\tau}(s) = \frac{\alpha - 1}{\lambda s} \quad \Rightarrow \quad \bigl( \dot{\tau}(s)\bigr)^{2} = \frac{(\alpha - 1)^{2}}{\lambda^{2} s^{2}}.
\]
We need, of course, to assume that $\alpha > 1$. Furthermore, we wish the coefficient attached to $V(x(s))$ to be $\frac{\alpha}{2s}$, i.e., 
\[
\bigl( \dot{\tau}(s)\bigr)^{2} \gamma(\tau(s)) = \frac{\alpha}{2s} \quad \Leftrightarrow \quad \gamma\left( \frac{\alpha - 1}{\lambda} \ln\left(\frac{s}{s_{0}}\right) + t_{0}\right) = \frac{\alpha}{2}\cdot\frac{\lambda^{2} s}{(\alpha - 1)^{2}}, 
\]
which is fulfilled if we choose 
\[
\gamma(t) := \frac{\alpha}{2}\cdot \frac{\lambda^{2} s_{0}}{(\alpha - 1)^{2}} \exp \left( \frac{\lambda (t - t_{0})}{\alpha - 1}\right). 
\]
Set $\mu(\cdot) = \gamma(\cdot)$. Notice that for every $t \geq t_0$
\[
\frac{\dot{\mu}(t)}{\gamma(t)} = \frac{\dot{\mu}(t)}{\mu(t)} = \frac{\left( \frac{\lambda s_{0}}{\alpha - 1}\right)^{2} \cdot \frac{\lambda}{\alpha - 1} \exp\left( \frac{\lambda (t - t_{0})}{\alpha - 1}\right)}{\left( \frac{\lambda s_{0}}{\alpha - 1}\right)^{2} \exp\left( \frac{\lambda (t - t_{0})}{\alpha - 1}\right)} = \frac{\lambda}{\alpha - 1}.
\]
Thus, the assumption \eqref{HBF:b-lambda monotone} is satisfied if and only if 
\[
\frac{\lambda}{\alpha - 1} < 1 \quad \text{and} \quad 2\lambda - 3 + \frac{\lambda}{\alpha - 1} > 0. 
\]
After rearranging terms, the two previous inequalities are equivalent to 
\begin{equation}\label{eq: inequality for lambda and alpha}
\frac{3(\alpha - 1)}{2\alpha - 1} = \frac{3}{2 + \frac{1}{\alpha - 1}} < \lambda < \alpha - 1.
\end{equation}
In particular, 
\[
\frac{3(\alpha - 1)}{2\alpha - 1} < \alpha - 1 \quad \Leftrightarrow \quad 2 < \alpha.
\]
We now turn our attention to the coefficient attached to $\frac{d}{ds}V(x(s))$ in \eqref{eq: time rescaled heavy ball, operator}. Since we want to reach \eqref{eq: fOGDA}, we need 
\[
\frac{\alpha}{2}\cdot \frac{\lambda}{\alpha - 1} = \dot{\tau}(s) \mu(\tau(s)) = 1 \quad \Leftrightarrow \quad \lambda = \frac{2(\alpha - 1)}{\alpha}.
\]
We must verify that this choice of $\lambda$ satisfies inequality \eqref{eq: inequality for lambda and alpha}. Indeed, 
\begin{gather*}
\frac{2(\alpha - 1)}{\alpha} < \alpha - 1 \quad \Leftrightarrow \quad 2 < \alpha \quad \text{and} \\
\frac{3(\alpha - 1)}{2\alpha - 1} < \frac{2(\alpha - 1)}{\alpha} \quad \Leftrightarrow \quad 3\alpha < 4\alpha - 2 \quad \Leftrightarrow \quad 2 < \alpha. 
\end{gather*}
All in all, $s\mapsto x(s)$ fulfills 
\[
\ddot{x}(s) + \frac{\alpha}{s}\dot{x}(s) + \frac{d}{ds}V(x(s)) + \frac{\alpha}{2s}V(x(s)) = 0. 
\]
Conversely, if for $\alpha > 2$, $x \colon [s_{0}, +\infty[ \to \mathcal{H}$ is a trajectory solution of the previous system and we define $y(t) := x(\sigma(t))$, where $\sigma \colon [t_{0}, +\infty[ \to [s_{0}, +\infty[$ is a continuously differentiable function such that $\dot{\sigma}(t) > 0$ for all $t\geq t_{0} \geq 0$ and $\lim_{t \to +\infty} \sigma(t) = +\infty$, arguing in a similar fashion as it was done previously we arrive at
\[
\ddot{y}(t) + \left[ \alpha \frac{\dot{\sigma}(t)}{\sigma(t)} - \frac{\ddot{\sigma}(t)}{\dot{\sigma}(t)}\right] \dot{y}(t) + \dot{\sigma}(t) \frac{d}{dt} V(y(t)) + \frac{\alpha}{2}\cdot \frac{\bigl( \dot{\sigma}(t)\bigr)^{2}}{\sigma(t)} V(y(t)) = 0. 
\]
We wish for the coefficient attached to $\dot{y}(t)$ to be $\lambda = \frac{2(\alpha - 1)}{\alpha}$. For this end, we need $\sigma$ to satisfy the differential equation
\[
\begin{cases}
\displaystyle \alpha\frac{\dot{\sigma}(t)}{\sigma(t)} - \frac{\ddot{\sigma}(t)}{\dot{\sigma}(t)} &= \displaystyle\frac{2(\alpha - 1)}{\alpha}, \\
\displaystyle\sigma(t_{0}) &= s_{0},
\end{cases}
\]
which is fulfilled by
\[
\sigma(t) := s_{0} \exp\left(\frac{2(t - t_{0})}{\alpha}\right).
\]
With this choice for $\sigma(\cdot)$, the resulting system reads
\[
\ddot{y}(t) + \frac{2(\alpha - 1)}{\alpha} \dot{y}(t) + \frac{2 s_{0}}{\alpha} \exp\left(\frac{2(t - t_{0})}{\alpha}\right)\frac{d}{dt} V(y(t)) + \frac{2 s_{0}}{\alpha}\exp\left( \frac{2(t - t_{0})}{\alpha}\right) V(y(t)) = 0.
\]
It is straightforward to check (we have actually done this already when going from the Heavy Ball system to the Fast OGDA dynamics) that
\[
\lambda = \frac{2(\alpha - 1)}{\alpha} \quad \mbox{and} \quad \mu(t) = \gamma(t) = \frac{2 s_{0}}{\alpha} \exp\left(\frac{2(t - t_{0})}{\alpha}\right) \quad \forall t \geq t_0
\]
satisfy Assumption \eqref{HBF:b-lambda monotone}, and thus all the results ensured by Theorem \ref{thm: heavy ball operator} hold.  

We have essentially shown the following proposition. 
\begin{proposition}\label{prop: heavy ball and AVD connected, operator}
Assume that $\alpha > 2$ and that $s_{0} >0, t_{0} \geq 0$ are initial times. Consider the following second-order systems:
\begin{equation}\label{eq: heavy ball and AVD connected, heavy ball}
\begin{cases}
\ddot{y}(t) + \frac{2(\alpha - 1)}{\alpha} \dot{y}(t) + \frac{2 s_{0}}{\alpha} \exp\left( \frac{2(t - t_{0})}{\alpha}\right) \frac{d}{dt} V(y(t)) + \frac{2 s_{0}}{\alpha} \exp\left( \frac{2(t - t_{0})}{\alpha}\right) V(y(t)) = 0, \\
y(t_{0}) = y_{0}, \quad \dot{y}(t_{0}) = y_{1},
\end{cases} 
\end{equation}
and 
\begin{equation}\label{eq: heavy ball and AVD connected, AVD}
\begin{cases}
\ddot{x}(s) + \frac{\alpha}{s} \dot{x}(s) + \frac{d}{ds}V(x(s)) + \frac{\alpha}{2s}V(x(s)) = 0, \\
x(s_{0}) = x_{0}, \quad \dot{x}(s_{0}) = x_{1}.
\end{cases}
\end{equation}
Then, the following statements are true:
 \begin{enumerate}[\rm (i)]  
        \item If $ y :[t_{0}, +\infty) \to \mathcal{H}$ is a solution trajectory to \eqref{eq: heavy ball and AVD connected, heavy ball} and the function $ \tau : [s_{0}, +\infty) \to [t_{0}, +\infty)$ is given by 
        \[
            \tau(s) := \frac{\alpha}{2} \ln\left( \frac{s}{s_{0}}\right) + t_{0}, 
        \]
        then the reparametrized trajectory $ x: [s_{0}, +\infty) \to \mathcal{H}$ given by $x(s) := y(\tau(s))$ is a solution to \eqref{eq: heavy ball and AVD connected, AVD} for initial conditions
        \[
            x(s_{0}) = y_{0} \quad \text{and} \quad \dot{x}(s_{0}) = \frac{\alpha}{2s_{0}} y_{1}.
        \]
        \item If $x : [s_{0}, +\infty) \to \mathcal{H}$ is a solution trajectory to \eqref{eq: heavy ball and AVD connected, AVD} and the function $ \sigma : [t_{0}, +\infty) \to [s_{0}, +\infty)$ is given by 
        \[
            \sigma(t) := s_{0}\exp\left( \frac{2(t - t_{0})}{\alpha}\right), 
        \]
        then the reparametrized trajectory $ y: [t_{0}, +\infty) \to \mathcal{H}$ given by $ y(t) := x(\sigma(t))$ is a solution to \eqref{eq: heavy ball and AVD connected, heavy ball} for initial conditions 
        \[
            y(t_{0}) = x_{0} \quad \text{and} \quad \dot{y}(t_{0}) = \frac{2 s_{0}}{\alpha} x_{1}. 
        \]
    \end{enumerate}
\end{proposition}
\subsection{Transferring the rates to Fast OGDA}

As a direct consequence of Theorem \ref{thm: heavy ball operator} and Proposition \ref{prop: heavy ball and AVD connected, operator}, we obtain the following theorem. 

\begin{theorem}
Let $\alpha > 2$, $s_0 >0$, $x_{*}$ be a zero of $V$, and $x : [s_{0}, +\infty[ \to \mathcal{H}$ be a solution trajectory of
\begin{equation*}
\begin{cases}
\ddot{x}(s) + \frac{\alpha}{s} \dot{x}(s) + \frac{d}{ds}V(x(s)) + \frac{\alpha}{2s}V(x(s)) = 0, \\
x(s_{0}) = x_{0}, \quad \dot{x}(s_{0}) = x_{1}.
\end{cases}
\end{equation*}
and let $x_{*}$ be a zero of $V$. Then, it holds
\[
\| V(x(s))\| = o\left( \frac{1}{s}\right), \quad \langle x(s) - x_{*}, V(x(s))\rangle = o\left(\frac{1}{s}\right), \quad \bigl\| \dot{x}(s)\bigr\| = o\left(\frac{1}{s}\right) \quad \mbox{as} \ s\to +\infty.
\]
Furthermore, $x(s)$ converges weakly to a zero of $V$ as $s\to +\infty$. 
\end{theorem}
\begin{proof}
 The proof is near identical to the one in the optimization setting. As per Proposition \ref{prop: heavy ball and AVD connected, operator}, define $ y(t) := x(\sigma(t))$. We know that $ y: [t_{0}, +\infty) \to \mathcal{H}$ is a solution to \eqref{eq: heavy ball and AVD connected, heavy ball}. Again, the function $ \sigma\circ\tau :[s_{0}, +\infty) \to [s_{0}, +\infty)$ is the identity, which gives $x(s) = y(\tau(s))$ and $\frac{2s}{\alpha} \dot{x}(s) = \dot{y}(\tau(s))$ for every $s \geq s_0$. According to Theorem \ref{thm: heavy ball operator}, if we set $ \mu(t) = \gamma(t) = \frac{2 s_{0}}{\alpha} \exp\left( \frac{2(t - t_{0})}{\alpha}\right)$ for all $t \geq t_0$, then it holds, as $s \to +\infty$, 
        \begin{equation*}
            \| V(x(s))\| = o\left( \frac{1}{\mu(\tau(s))}\right), \quad \langle x(s) - x_{*}, V(x(s))\rangle = o \left( \frac{1}{\mu(\tau(s))}\right) \quad \text{and} \quad \frac{2s}{\alpha} \bigl\|\dot{x}(s)\bigr\| \to 0
        \end{equation*}
        We just need to compute $ \mu(\tau(s))$. We readily see
        \[
            \mu(\tau(s)) = \mu\left( \frac{\alpha}{2}\ln\left(\frac{s}{s_{0}}\right) + t_{0}\right) = \frac{2 s_{0}}{\alpha} \exp\left(\frac{2}{\alpha}\cdot \frac{\alpha}{2} \ln\left(\frac{s}{s_{0}}\right)\right) = \frac{2 s}{\alpha}, 
        \]
which immediately shows the rates claimed in the statement. Since $y(t)$ converges weakly to a zero of $V$ as $t\to +\infty$, so does $x(s) = y(\tau(s))$ as $s\to +\infty$.
\end{proof}

\appendix

\section{Auxiliary results} In what follows, we briefly recall some of the results used throughout the paper.

A key tool in proving the weak convergence of the solution trajectory is Opial's Lemma.

\begin{lemma}\label{Opial} (Opial) Let $S$ be a nonempty subset of $\mathcal H$ and let $x:[t_0,+\infty[\to \mathcal H$. Assume that 
\begin{enumerate}[\rm (i)]
\item for every $x_{*} \in S$, $\lim_{t\to +\infty} \left\lVert x(t)-x_{*} \right\rVert$ exists;
\item every weak sequential limit point of $x(t)$, as $t \to +\infty$, belongs to $S$.
\end{enumerate}
Then $x(t)$ converges weakly as $t\to +\infty$, and its limit belongs to $S$.
\end{lemma}

The following result can be found in \cite[Lemma 5.1]{AAS}.
\begin{lemma}
\label{lem:lim-R}
Let $\delta > 0$.
Suppose that $F \colon \left[ \delta , + \infty \right) \to \R$ is locally absolutely continuous, bounded from below, and there exists $G \in \sL^{1} \left( \left[ \delta , + \infty \right); \R \right)$ such that for almost every $t \geq \delta$
\begin{equation*}
\dfrac{d}{dt} F \left( t \right) \leq G \left( t \right) .
\end{equation*}
Then the limit $\lim\limits_{t \to + \infty} F \left( t \right) \in \R$ exists.
\end{lemma}

\begin{lemma}
\label{lem:quad}
Let $A, B, C \in \R$ be such that $A \neq 0$ and $B^{2} - AC \leq 0$.
Then, the following statements are true:
\begin{enumerate}[\rm (i)]
\item
\label{quad:vec-pos}
if $A > 0$, then it holds
\begin{equation*}
A \left\lVert x \right\rVert ^{2} + 2B \left\langle x , y \right\rangle + C          \left\lVert y \right\rVert ^{2} \geq 0 \quad \forall x, y \in \mathcal{H};
\end{equation*}
\item
\label{quad:vec}
if $A < 0$, then it holds
\begin{equation*}
A \left\lVert x \right\rVert ^{2} + 2B \left\langle x , y \right\rangle + C          \left\lVert y \right\rVert ^{2} \leq 0 \quad \forall x, y \in \mathcal{H}.
\end{equation*}
\end{enumerate}
\end{lemma}

The following lemma appears as Lemma A.3 in \cite{fOGDA with r} and generalizes Lemma A.2 from \cite{APR2}.

\begin{lemma}\label{lem:q}
Let $a > 0$, $r \in [0,1]$ and $q \colon [t_{0}, +\infty) \to \R$ be a continuously differentiable function such that 
\[
\lim_{t\to +\infty} \left( q(t) + \frac{t^r}{a} \dot{q}(t)\right) = \ell\in \R. 
\]
Then it holds $\lim_{t\to +\infty} q(t) = \ell$. 
\end{lemma}

 \section*{Acknowledgments}
This work was started in the last year of Hedy Attouch’s life, just months before his passing. R.I. Bo\c t, D. A. Hulett and D.-K. Nguyen would like to take this opportunity to pay tribute to a remarkable mathematician, a kind and generous soul, whose absence is deeply felt.


\end{document}